\documentclass[final]{siamart}
\usepackage{lipsum}
\usepackage{amsfonts}
\usepackage{graphicx}
\usepackage{algorithmic}
\usepackage{amsmath}
\usepackage{amssymb}
\usepackage{subfig}
\usepackage{booktabs}
\usepackage{multirow}
\usepackage{amsopn}

\graphicspath{{draw/}}
\usepackage{subfig}
\usepackage{color}
\usepackage[utf8x]{inputenc}

\begin{document}
\headers{Energy stable schemes for gradient flows based on the DVD method}{Jizu Huang}

\title{Energy stable schemes for gradient flows based on the DVD method
}

\author{Jizu Huang\thanks{LSEC, Academy of Mathematics and Systems Science, Chinese Academy of Sciences, Beijing 100190, China, and School of Mathematical Sciences, University of Chinese Academy of Sciences, Beijing 100049, China
		(\email{huangjz@lsec.cc.ac.cn}). The work was supported by the NSFC 11871069 and Strategic Priority Research Program of the Chinese Academy of Sciences XDC06030101.}
}

\maketitle

\begin{abstract}
The existing discrete variational derivative method is only second-order accurate and fully implicit. In this paper, we propose a framework to construct an arbitrary high-order implicit (original) energy stable scheme and a second-order semi-implicit (modified) energy stable scheme. Combined with the Runge--Kutta process, we can build an arbitrary high-order and unconditionally (original) energy stable scheme based on the discrete variational derivative method. The new energy stable scheme is implicit and leads to a large sparse nonlinear algebraic system at each time step, which can be efficiently solved by using an inexact Newton type algorithm.  To avoid solving nonlinear algebraic systems, we then present a relaxed discrete variational derivative method, which can construct second-order, linear, and unconditionally (modified) energy stable schemes. Several numerical simulations are performed to investigate the efficiency, stability, and accuracy of the newly proposed schemes.
\end{abstract}



	\section{Introduction}\label{sec-intro}
	
	Many physical problems, such as interface dynamics \cite{anderson1998diffuse, liu2003phase, lowengrub1998quasi}, 
	crystallization \cite{elder2002modeling, elder2007phase, wei2019parallel}, thin  films \cite{giacomelli2001variatonal},  
	polymers \cite{fraaije1997dynamic, fraaije1993dynamic, maurits1997mesoscopic}, binary alloys \cite{cahn1994evolution, huang2020parallel, zhang2016extreme}, 
	and liquid crystals \cite{forest2004weak, leslie1979theory}, 
	can be modeled by gradient flow type partial differential equations (PDEs). 
	Gradient flows  are driven by free energy functional and can be written in the following general form:
	\begin{equation}
	\label{system-1}
	\frac{\partial \phi}{\partial t} = {\cal G}\frac{\delta{\cal E}}{\delta \phi}\quad \hbox{in}~\Omega\times(0, {T}],
	\end{equation}
	where $\phi:=\phi(\mathbf x,t)$, $\Omega\in\mathbb{R}^d$ with $d=1,\,2,$ or $3$ is the computational domain, $\cal E[\phi]$ is a given total free energy functional, and $\frac{\delta{\cal E}}{\delta \phi}$ is the variational derivative. 
	The free energy functional usually can be written explicitly as
	\begin{equation}
	\label{free-energy}
	{\cal E}[\phi]:=\frac {\gamma}2  \left<\nabla \phi,\,\nabla \phi\right>+{\cal E}_1(\phi),
	\end{equation}
	where $\left<\varphi,\,\psi\right>=\int_\Omega \varphi\psi\textnormal{d} \mathbf x$, $\gamma$ is an  interface energy coefficient, and ${\cal E}_1(\phi)$ is a body free energy functional.
	To simplify the presentation, we  only consider the well-known  Ginzburg–Landau free energy functional
	\begin{equation}
	\label{local-free-energy}
	{\cal E}_1(\phi)=\left<{ E}_1(\phi)\right>:=\int_\Omega \frac 1 {4\epsilon^2}(1-\phi^2)^2\hbox{d}\mathbf x,
	\end{equation}
	where $\epsilon$ is a small parameter. 
	
	In \eqref{system-1}, the operators ${\cal G}$ are  set as $-I$, $\Delta$, and $-(-\Delta)^\alpha$ with $(0<\alpha<1)$  for dissipation mechanisms including $L^2$, $H^{-1}$, and $H^{-\alpha}$ gradient flows, respectively. Assuming ${\cal G}$ is nonpositive, the gradient flow system satisfies the following energy dissipation law:
	\begin{equation}
	\frac{\textnormal{d}\cal E[\phi]}{\textnormal{d} t}=\left<\frac{\delta{\cal E}}{\delta \phi}\cdot \frac{\partial \phi}{\partial t} \right>
	=\left<\frac{\delta{\cal E}}{\delta \phi},\, {\cal G}\frac{\delta{\cal E}}{\delta \phi}\right>\leq 0.
	\end{equation}
	To simplify the presentation, we use periodic boundary conditions or homogeneous Neumann boundary conditions
	such that all boundary terms will vanish when integration by parts is performed. 
	
	For a gradient flow, it is very important to construct an efficient and accurate numerical method, which preserves the energy dissipation law at the discrete level.
	There exist several popular approaches, such as the convex splitting method \cite{baskaran2013convergence, elliott1993global,  eyre1998unconditionally, shen2012second}, 
	the stabilization method \cite{shen2010numerical, zhu1999coarsening}, the
	exponential time differencing \cite{ju2015fast, zhang2016extreme}, the invariant energy quadratization (IEQ) \cite{yang2016linear, zhao2017numerical}, the scalar auxiliary variable (SAV) \cite{shen2018scalar, shen2019new}, the discrete variational derivative (DVD) method \cite{du1991numerical, huang2020parallel, huang2022parallel, DVDM, PFCDVD}, and so on.
	Most of them are implicit and semi-implicit second-order schemes. 
	In contrast to fully implicit methods, semi-implicit methods only require solving a linear system at each time step, which is easier to implement and less computationally expensive.  
	In recent years, many researchers focus on designing higher-order energy stable schemes, such as \cite{akrivis2019energy, gong2020arbitrarily, gong2020arbitrarily, lee2022high, shin2017convex, yang2022arbitrarily}.
	Without any additional assumptions, the DVD method, originally proposed in \cite{du1991numerical}, can build a second-order implicit scheme for gradient flows with complicated total free energy functional, such as total free energy functional of logarithmic type \cite{huang2020parallel} and including elastic energy \cite{huang2022parallel}. 
	However, the construction of implicit energy stable schemes of order higher than 2 or semi-implicit linear schemes  based on the DVD method has remained open.

	For general gradient systems, usually written as ordinary differential equations, the discrete gradient method is proposed to preserve the first integral \cite{dahlby2011preserving, hairer2014energy, mclachlan1999geometric, norton2013discrete}. 
	By combining the discrete gradient method with the implicit Runge--Kutta process, an arbitrary higher-order scheme can be established, which preserves the first integral.
	In this paper, by combining the DVD method and  the implicit Runge--Kutta process, we present a framework to construct an arbitrary higher-order numerical scheme for the general phase field systems \eqref{system-1}. 
	Since the gradient flow is a PDE system satisfying energy dissipation law, the framework based on the DVD method is different from the discrete gradient method.
	In the DVD method, the discrete variational derivative $\mu[\phi_i,\phi_j]$ for two discrete total free energy functionals ${\cal E}^i:={\cal E}[\phi_i]$ and ${\cal E}^j:={\cal E}[\phi_j]$ is defined as a solution of the following equation
	\begin{equation}
	\label{eq:1}
	 {\cal E}^i-{\cal E}^j=\left<\phi_i-\phi_j,\,\mu[\phi_i,\phi_j]\right>,
	\end{equation}
	where $\phi_i,\,\phi_j$ are two solutions of gradient flow \eqref{system-1} at different times.
	With the definition of the discrete variational derivative, we can construct an arbitrary higher-order numerical scheme by following the idea of the implicit Runge--Kutta process.
	The newly proposed scheme can be proved to be unconditionally stable and satisfy the energy dissipation law if a sufficient condition is satisfied. 
	Due to the nonlinearity of the free energy functional ${\cal E}$, the arbitrary higher-order numerical scheme built by using the framework is fully implicit, which is hard to implement and computationally expensive compared to linear schemes.
	To construct a semi-implicit linear scheme, we should choose the discrete variational derivative $\mu[\phi_i,\phi_j]$ as a linear functional of $\phi_i$. 
	By  relaxing the formula \eqref{eq:1}, we propose a relaxed DVD method to build a semi-implicit linear scheme. 
	As compared with  the IEQ method and the SAV method, the newly proposed relaxed DVD method is more efficient.
	
	The remainder of this paper is organized as follows. 
	In Sec. 2, we introduce a framework to  construct an arbitrary higher-order implicit unconditionally energy stable scheme. 
	In Sec. 3, we propose the relaxed DVD method to build semi-implicit linear scheme. 
	In Sec. 4, we discuss the fully discrete scheme and corresponding solvers. 
	Several numerical simulations are reported in Sec. 5 and the paper is concluded in Sec. 6.

	\section{A framework for constructing high-order energy stable schemes}
	In this section, we introduce a framework to  construct an arbitrary higher-order, implicit, and unconditionally energy stable scheme for general gradient flows \eqref{system-1} by combining the  the DVD method and  the implicit Runge--Kutta process.
	It is important to point out that the framework proposed in this section is different from the discrete gradient method  \cite{dahlby2011preserving, hairer2014energy, mclachlan1999geometric, norton2013discrete}.
	
	\subsection{ Runge--Kutta Processes}
	In this subsection, we briefly review the  Runge--Kutta processes \cite{butcher1963coefficients, butcher2016numerical}, which can construct a temporal discretization of the gradient flow system \eqref{system-1}.
	Let us rewrite the gradient flow system  \eqref{system-1} to be an  initial value problem specified as follows:
	\begin{equation}
	\label{system-2}
	\frac{\partial \phi}{\partial t} = {\cal G}\frac{\delta{\cal E}}{\delta \phi}:=f(\phi,t)\quad \hbox{with}~\phi(t_0)=\phi_0.
	\end{equation}	
	For a given time step size $h>0$, a $\nu$ stage Runge--Kutta Process obtains an approximate solution $\phi_{\nu}$ at $t=t_0+h$  by using the following equations:
	\begin{equation}
	\label{system-3}
	\left\{
	\begin{aligned}
	 \phi_{\nu}&= \phi_{0}+ h \sum\limits_{i=1}^\nu b_i   k_{i},\\
	 k_i&=f\left(t_0+c_i h,\phi_0+ h\sum\limits_{j=1}^\nu a_{ij} k_j\right)\qquad i=1,\,2,\,\cdots,\, \nu,\\
	\end{aligned}
	\right.
	\end{equation}	
	where $a_{ij},\,b_i$, $(i,\,j=1,\,2,\,\cdots,\, \nu)$ are numerical constants. For convenience we shall designate the process by an array as the following Butcher tableau:
	\begin{equation}
	\begin{array}{cccc|c}
 	 a_{11} & a_{12}  & \cdots &a_{1\nu}  & c_{1} \\
 	 a_{21} & a_{22}  & \cdots &a_{2\nu}  & c_{2} \\
 	 \vdots & \vdots  &  &\vdots & \vdots \\
 	 a_{\nu1} & a_{\nu2}  & \cdots &a_{\nu\nu}  & c_{\nu} \\
  	\hline
  	b_1 & b_2 & \cdots & b_\nu&
	\end{array}
	\end{equation}
	with $c_i=\sum\limits_{j=1}^\nu a_{ij} $. 
	The Runge--Kutta Process \eqref{system-3} is called ``semi-implicit" if $a_{ij}=0$ for $i<j$.
	In particular, if  additionally $a_{ii}=0$, it is called ``explicit". In contrast to the ``semi-implicit" and ``explicit" processes, the Runge--Kutta Process \eqref{system-3} is called ``implicit" if there exists $a_{ij}\neq 0$ for $i<j$.
	The numerical constants $a_{ij},\,b_i$, and $c_i$  can be determined by following the integration process of evaluating $\int\limits_{t_0}^{t_{0}+c_ih}f(\phi)\textnormal d t$, such as the Gauss--Legendre quadrature formulae and the Newton--Cotes formulae \cite{quarteroni2010numerical}. 
%
%
%
%
	
	\subsection{An arbitrary high-order and unconditionally energy stable  DVD scheme}
	It is easy to build an arbitrary high-order scheme for the gradient flow   \eqref{system-1} by using the Runge--Kutta process.  
	However,  most of them are not unconditionally energy stable. 
	To overcome the difficulty, we propose a framework to construct arbitrary high-order and unconditionally energy stable schemes based on the DVD method.
	The DVD method was originally proposed in \cite{du1991numerical}, and then used to solve phase field crystal equation \cite{PFCDVD}, AC/CH equations \cite{huang2020parallel}, and Ni-based alloys system \cite{huang2022parallel}.
	The DVD methods in \cite{du1991numerical, huang2022parallel, huang2020parallel, du1991numerical}  are only second-order accurate in time.
	It is non-trivial  to construct a high-order scheme by using the DVD method.
	To do it, let us first introduce some necessary notations and definitions. 
	For given constants $0\leq c_i\leq1$ $(i=1,\,\cdots,\, \nu)$, let us denote $\phi_i\approx \phi(\mathbf{x},\,t_0+c_i h)$ as the approximate solutions for the gradient flow   \eqref{system-1}  at $t=t_0+c_i h$. The corresponding discrete total free energy is denoted as ${\cal E}^i:={\cal E}[\phi_i]$.
	
	For two discrete total free energy functionals ${\cal E}^i$ and  ${\cal E}^j$, we define the discrete variational derivative $\mu[\phi_i,\phi_j]$  as a solution of the following equation
	\begin{equation}
	\label{d-v-d}
	 {\cal E}^i-{\cal E}^j=\left<\phi_i-\phi_j,\,\mu[\phi_i,\phi_j]\right>.
	\end{equation}
	According to the definition of total free energy \eqref{free-energy}, we can derive the discrete variational derivative $\mu[\phi_i,\phi_j]$ as
	\begin{equation}
	 \mu[\phi_i,\phi_j]=-\frac \gamma 2 \Delta\left( \phi_i+ \phi_j\right)+E_1\{\phi_i,\phi_j\},
	\end{equation}
	where $ E_1\{\phi_i,\phi_j\}:= \frac{E_1[\phi_i]-E_1[\phi_j]}{\phi_i - \phi_j}$ is the first-order finite quotient of functional $E_1$.
	If the energy functional $E_1$ is a polynomial function of $\phi$,  the first-order finite quotient $ E_1\{\phi_i,\phi_j\}$ is also a polynomial function of $\phi_i$ and $\phi_j$.
	For the energy functional $E_1$ containing a non-polynomial function of $\phi$, such as logarithmic type functional,  the first-order finite quotient $ E_1\{\phi_i,\phi_j\}$ is no longer a polynomial function of $\phi_i$ and $\phi_j$.
	In this situation, the numerical computation  of the first-order finite quotient  $E_1\{\phi_i,\phi_j\}$ is unstable and inaccurate when $\phi_i-\phi_j$ is close to zero. 
	Unfortunately, in phase field systems, $\phi_i-\phi_j$ is often close to zero,  which indicates that the numerical calculation of $ E_1\{\phi_i,\phi_j\}$ is numerically unstable and inaccurate. 
	To overcome the difficulty,  a stable and highly accurate approximation for the first-order finite quotient $ E_1\{\phi_i,\phi_j\}$ was proposed in  \cite{huang2020parallel, huang2022parallel}.
	For ease of presentation, we focus on the double well potential, that is $E_1(\phi)= \frac 1 {4\epsilon^2}(1-\phi^2)^2$.  
	The first-order finite quotient $ E_1\{\phi_i,\phi_j\}:= \frac{E_1[\phi_i]-E_1[\phi_j]}{\phi_i - \phi_j}$ is defined as 
	\begin{equation}
	\label{nonlinear-1}
	 E_1\{\phi_i,\phi_j\}:= \frac{E_1[\phi_i]-E_1[\phi_j]}{\phi_i - \phi_j}=\frac 1 {4\epsilon^2}(\phi_i+\phi_j)(\phi^2_i+\phi^2_j-2).
	 \end{equation}

	Let us denote a vector $\mathbf {y}=(\mu[\phi_1,\phi_0],\,\mu[\phi_2,\phi_0],\,\cdots,$ $\mu[\phi_{\nu},\phi_{\nu-1}])^T\in \mathbb{R}^{\hat \nu}$ 
	and $ \overrightarrow{\delta\phi}=(\phi_1-\phi_0,\,\phi_2-\phi_0,\,\cdots,\,\phi_{\nu}-\phi_{\nu-1})^T\in \mathbb{R}^{\hat \nu}$ with $\hat \nu=\frac{\nu(\nu+1)}{2}$.	
	Following the idea of the Rung--Kutta process, we define a $\nu$ stage DVD scheme as follows:
	\begin{equation}
	\label{system-4}
	\begin{aligned}
	 \phi_i= \phi_0+ h \sum\limits_{j=1}^{\hat\nu}\tilde{ a}_{ij} {\cal G} y_j \qquad i=1,\,2,\,\cdots,\, \nu,
	\end{aligned}
	\end{equation}
	where  $\tilde a_{ij}$, $(i=1,\,2,\,\cdots,\, \nu,\, j=1,\,2,\,\cdots,\, \hat\nu)$ are numerical constants. For convenience we shall designate the process by an array as the following tableau:
	
	\begin{equation}
	\label{tableau-dvd}
	\begin{array}{cccc|c}
 	 \tilde a_{11} & \tilde  a_{12}  & \cdots &\tilde  a_{1\hat\nu}  & c_{1} \\
 	 \tilde  a_{21} &\tilde  a_{22}  & \cdots &\tilde  a_{2\hat\nu}  & c_{2} \\
 	  \vdots & \vdots  &  &\vdots & \vdots \\
 	\tilde  a_{\nu1} & \tilde  a_{\nu2}  & \cdots &\tilde  a_{\nu\hat\nu}   & c_{\nu},
	\end{array}
	\end{equation}
	where $c_i=\sum\limits_{j=1}^{\hat\nu}\tilde a_{ij} $. 
	Similar to the derivation of the Runge--Kutta method \cite{butcher2016numerical},  we can obtain the numerical constants $\tilde a_{ij}$ for given $c_i$ and $\nu$.
	It is important to point out that the numerical constants $\tilde a_{ij}$ should be chosen carefully such that the proposed scheme \eqref{system-4} is unconditionally energy stable.  
	More details on how to choose the numerical constants $\tilde a_{ij}$ are given in Appendix A.
	For ease of presentation, let us rewrite the scheme \eqref{system-4} in the matrix form as follows
	\begin{equation}
	\label{system-4-1}
	\begin{aligned}
	\overrightarrow{\delta \phi}= h {\cal G}{\mathbb A} \mathbf{y},
	\end{aligned}
	\end{equation}
	where $\mathbb{A}=(\tilde a_{ij})$ is a $\hat\nu$-order square matrix.
	The $\nu+1,\, \cdots, \hat\nu$ rows of the matrix $\mathbb{A}$ are computed  by using tableau \eqref{tableau-dvd}.
	Next, we call a vector $\mathbf{v}=(v_1,\,\cdots,\,v_{\hat\nu})\in \mathbb{R}^{\hat\nu}$ as a unit partition vector of ${\cal E}^\nu-{\cal E}^0$ if and only if 
	\begin{equation}
	{\cal E}^\nu-{\cal E}^0=v_1 ({\cal E}^1-{\cal E}^0)+ v_2 ({\cal E}^2-{\cal E}^0)+\cdots + v_{\hat \nu} ({\cal E}^{\nu}-{\cal E}^{\nu-1}).
	\end{equation}
	According to the definition of the discrete variational derivative \eqref{d-v-d}, we have
	\begin{equation}\label{eq:2-10}
	\begin{aligned}
	{\cal E}^\nu-{\cal E}^0&=\frac 1 2\left<\mathbf{y},\,\hbox{diag}(\mathbf{v})\overrightarrow{\delta \phi}\right>+\frac 12 \left<\overrightarrow{\delta \phi},\, \hbox{diag}(\mathbf{v})\mathbf{y}\right>
	\\
	&=h\left<\mathbf{y},\, {\cal G}\mathbb{B}\mathbf{y}\right>,
	\end{aligned}
	\end{equation}
	where $\mathbb{B}=\frac 1 2\left(\hbox{diag}(\mathbf{v})\mathbb{A}+\mathbb{A}^T\hbox{diag}(\mathbf{v})\right)\in\mathbb{R}^{\hat\nu\times\hat\nu}$ is a symmetric matrix.

   	The following theorem gives a sufficient condition such that the $\nu$ stage DVD scheme is unconditionally energy stable.
	\begin{theorem}\label{theorem:0}
	For any given time step $h >0$, if there exists a unit partition vector $\mathbf{v}$ of ${\cal E}^\nu-{\cal E}^0$ such that  the matrix $\mathbb{B}$ is positive semidefinite,
	 the $\nu$ stage DVD scheme is unconditionally energy stable and enjoys the following energy dissipation law
	\begin{equation}
	{\cal E}^\nu\leq{\cal E}^0.
	\end{equation}
	\end{theorem}
	\begin{proof}
	For a  positive semidefinite  matrix $\mathbb{B}$ and a negative semidefinite scale operator ${\cal G}$, it follows from \eqref{eq:2-10} that
	\begin{equation}
	{\cal E}^\nu-{\cal E}^0=h\left<\mathbf{y},\, {\cal G}\mathbb{B}\mathbf{y}\right> \leq 0
	\end{equation}
	holds for any vector $\mathbf{y}$. 
	The unconditional energy stability and  the  energy dissipation law of the $\nu$ stage DVD scheme can be  obtained from this formula. 
	\end{proof}
	The following are some $\nu$  stage schemes obtained by using the newly proposed framework,  where $p$ denotes the order of convergence rate.
		
		\begin{equation}
		\label{Sch-1}
		\renewcommand{\arraystretch}{1.5}
		\hbox{Sch-1},\qquad(\nu = 1,\,p=2), \qquad
	\begin{array}{c|c}
 	1 & 1. 
	\end{array}
	\end{equation}
	
			\begin{equation}\label{Sch-2}
		\renewcommand{\arraystretch}{1.5}
		\hbox{Sch-2},\qquad(\nu = 2,\,p=3), \qquad
	\begin{array}{ccc|c}
	  \frac 7 {18}  &-\frac 1 {6} & \frac 1 {9} & \frac 1 3 \\
 	  \frac 1 2& -\frac 1 2 &  1 & 1. 
	\end{array}
	\end{equation}
	
		\begin{equation}\label{Sch-3}
		\renewcommand{\arraystretch}{1.5}
		\hbox{Sch-3},\qquad(\nu = 2,\,p=4), \qquad
	\begin{array}{ccc|c}
 	  \frac 7 {12}  &-\frac 1 {6} & \frac 1 {12} & \frac 1 2 \\
 	  \frac 2 3& -\frac 1 3 & \frac 2 3 & 1. 
	\end{array}
	\end{equation}
	
			\begin{equation}\label{Sch-4}
		\renewcommand{\arraystretch}{1.5}
		\hbox{Sch-4},\qquad(\nu = 3,\,p=4),  \qquad
	\begin{array}{cccccc|c}
 	\frac{25}{72}& 0& - \frac 1 {24}&\frac {1} {72}  & 0 &\frac 1 {72}  & \frac 1 3 \\
 	 \frac {13}{ 36}&0  &-\frac 1{12}&\frac {13}{36}& 0 & \frac 1 {36} & \frac 2 3 \\
 	 \frac 3 8& 0 &-\frac 1 8&\frac 3 8 & 0 & \frac 3 8& 1. 
	\end{array}
	\end{equation}

	\begin{theorem}\label{theorem:1}
	For any given time step $h >0$,  the numerical schemes named as Sch-1, Sch-2, Sch-3, and Sch-4
	are unconditionally energy stable and  satisfy the following energy dissipation law
	$${\cal E}^\nu\leq {\cal E}^0.$$
	\end{theorem}
	\begin{proof} We complete the proof of this theorem by choosing a unit partition vector $\mathbf{v}$ such that the matrix $\mathbb{B}$ is positive semidefinite. 
	
	For the scheme Sch-1, we set $\mathbf{v}=1$ and the corresponding $\mathbb{B}=1$.
	
	For the scheme Sch-2, we choose $\mathbf{v}=(\frac 3 2,\, -\frac 1 2,\, \frac 3 2)^T$. The resulting positive semidefinite matrix $\mathbb{B}$  is
			\begin{equation}
		\renewcommand{\arraystretch}{1.}
		\mathbb{B} =\frac 1 {12} \left(
	\begin{array}{ccc}
	7 & -3 & 2\\
	-3& 3 & -6\\
	2 & -6 & 48
	\end{array}
	\right).
	\end{equation}
	
	For the scheme Sch-3, by set $\mathbf{v}=(\frac 4 3,\, -\frac 1 3,\, \frac 4 3)^T$, we have the positive semidefinite matrix $\mathbb{B}$ as follows
			\begin{equation}
		\renewcommand{\arraystretch}{1.}
		\mathbb{B} =\frac 1 9 \left(
	\begin{array}{ccc}
 	7 & -2 & 1 \\
 	-2 & 1 & -2\\
 	1 & -2 & 7
	\end{array}
	\right).
	\end{equation}
	
		For the scheme Sch-4, we choose $\mathbf{v}=\frac 1 8(9,\,0,\, - 1 ,\,  9 ,\,0,\,9 )^T$. The resulting positive semidefinite matrix $\mathbb{B}$ is
			\begin{equation}
		\renewcommand{\arraystretch}{1.}
		\mathbb{B} =\frac 1 {288} \left(
	\begin{array}{cccccc}
 	 {225}& 0 &  -27 & 9 & 0 & 9 \\
 	 0& 0 &  0 & 0 & 0 & 0 \\
 	 -27& 0 & 9 & -27 & 0 & -27 \\
 	 9& 0 &  -27 & 225 & 0 & 9 \\
 	 0& 0 &  0 & 0 & 0 & 0 \\
 	 9& 0 &  -27 & 9 & 0 & 225 \\
	\end{array}
	\right).
	\end{equation}
	\end{proof}
	
	The framework introduced in this section can construct arbitrary high-order  DVD schemes.
	We prove that the DVD scheme is unconditionally energy stable and satisfies  the original energy dissipation law without any assumptions and stabilization parameters, 
	which should be a big advantage of the newly proposed DVD scheme. 
        Due to the nonlinearity of the total free energy, we have that the discrete variational derivative given in \eqref{nonlinear-1} is a nonlinear function of $\phi_i$. 
        Thus, the Sch-1, Sch-2, Sch-3, and Sch-4 schemes are nonlinear and fully implicit schemes, which lead to nonlinear algebraic systems when fully discretized.
        However, it is impossible to construct a linear semi-implicit or explicit unconditionally energy stable scheme by  using the DVD method. 
        Compared with the well known linear schemes, such as the IEQ \cite{yang2016linear, zhao2017numerical} and the SAV \cite{shen2018scalar, shen2019new} approaches,  those nonlinear implicit schemes are more complex and less efficient.
        To overcome this difficulty, we proposed several relaxed DVD methods in the next section to construct linear unconditionally energy stable schemes for gradient flow problems.
	
	\section{The relaxed DVD method}
	In this section, we propose several relaxed DVD methods by combining ideas from existing linear approaches.
	The newly proposed method is called the relaxed DVD method because the formula \eqref{d-v-d} for the discrete total free energy functionals and the discrete variational derivative is relaxed to construct a linear unconditionally energy stable scheme. 
	
	The first type of  relaxed DVD method can be constructed by relaxing the formula \eqref{d-v-d} to
	\begin{equation}
	\label{d-v-d-relax}
	\begin{aligned}
	 {\cal E}^i-{\cal E}^j\leq\left<\phi_i-\phi_j,\,\mu[\phi_i,\phi_j]\right>,
	\end{aligned}
	\end{equation}
	where the discrete variational derivative $\mu[\phi_i,\phi_j]$ should be chosen as a linear function of $\phi_i$ for $i>j$.
	We then define a $\nu$ stage relaxed DVD scheme as follows:
	\begin{equation}
	\label{relaxed-dvd-scheme-1}
	\begin{aligned}
	 \phi_i= \phi_0+ h \sum\limits_{j=1}^{\hat\nu}\tilde{ a}_{ij} {\cal G} y_j \qquad (i=1,\,2,\,\cdots,\, \nu),
	\end{aligned}
	\end{equation}
	where $y_{\frac {j(2\nu-j+1)}{2}+(i-j)}=\mu[\phi_i,\phi_j]$.
	The $\nu$ stage relaxed DVD scheme \eqref{relaxed-dvd-scheme-1} is a linear scheme due to the definition of the discrete variational derivative \eqref{d-v-d-relax}.
	For a non-negative unit partition vector  $\mathbf{v}$  of ${\cal E}^\nu-{\cal E}^0$,
	we have
	\begin{equation}
	\begin{aligned}
	{\cal E}^\nu-{\cal E}^0&\leq\frac 1 2\left<\mathbf{y},\,\hbox{diag}(\mathbf{v})\overrightarrow{\delta \phi}\right>+\frac 12 \left<\overrightarrow{\delta \phi},\, \hbox{diag}(\mathbf{v})\mathbf{y}\right>
	\\
	&=h\left<\mathbf{y},\, {\cal G}\mathbb{B}\mathbf{y}\right>.
	\end{aligned}
	\end{equation}
	If the matrix $\mathbb{B}$ is positive semidefinite, it is straightforward to prove that the $\nu$ stage relaxed DVD scheme is unconditionally energy stable.
	
	Next, we introduce an approach to select $\mu[\phi_i,\phi_j]$ satisfying \eqref{d-v-d-relax} by following the idea of 
	the stabilization method \cite{shen2010numerical, zhu1999coarsening}.
	For any simple linear operator $\hat{\cal L}$ such that both $\hat{\cal L}$ and $\hat{\cal L}-\frac {\delta ^2 {\cal E}_1}{(\delta \phi)^2}[\phi]$ are positive, we take a particular convex splitting as ${\cal E}[\phi]:={\cal E}_c[\phi]-{\cal E}_e[\phi]$ with 
	$$ {\cal E}_c[\phi]=-\frac \gamma 2\left<\phi,\,\Delta\phi\right>+\frac 1 2\left<\phi,\,\hat{\cal L}\phi\right>,\quad {\cal E}_e[\phi]=\frac 1 2\left<\phi,\,\hat{\cal L}\phi\right>-{\cal E}_1[\phi],$$
	where  both ${\cal E}_c$ and ${\cal E}_e$ are convex about unknown $\phi$.
	The linear operator $\hat{\cal L}$ is commonly chosen as 
	$$
	\hat{\cal L}=\hat a_0+\hat a_1(-\Delta)+\hat a_2(-\Delta)^2+\cdots.
	$$
	For $\nu=1$, we can take the discrete variational derivative $\mu[\phi_1,\phi_0]$
	as $\frac{\delta {\cal E}_c}{\delta\phi}[\phi_1] - \frac{\delta {\cal E}_e}{\delta\phi}[\phi_0]$.
	Due to the property of the convex functional
	$$
	{\cal E}_c[\phi_1]-{\cal E}_c[\phi_0]\geq\left< \frac{\delta {\cal E}_c}{\delta\phi}[\phi_0],\phi_1-\phi_0\right>,
	$$
	we get the following inequality
	\begin{equation}
	\label{d-v-d-cvs}
	\begin{aligned}
	 {\cal E}^1-{\cal E}^0:&={\cal E}_c[\phi_1]-{\cal E}_c[\phi_0]-({\cal E}_e[\phi_1]-{\cal E}_e[\phi_0])\\
	 &\leq\left< \frac{\delta {\cal E}_c}{\delta\phi}[\phi_1],\phi_1-\phi_0\right>-\left< \frac{\delta {\cal E}_e}{\delta\phi}[\phi_0],\phi_1-\phi_0\right>\\
	 &=\left<\phi_1-\phi_0,\,\mu[\phi_1,\phi_0]\right>,
	\end{aligned}
	\end{equation}
	which leads to the well known unconditionally energy stable scheme (stabilization method):
	\begin{equation}
	\label{d-v-d-cvs-s}
	\phi_1-\phi_0= h{\cal G}\mu[\phi_1,\phi_0].
	\end{equation}
	The stabilization method is first-order and can be extended to second-order schemes, but in general, it cannot be unconditionally energy stable.
	The relaxed DVD method gives a potential way to construct high-order unconditionally energy stable scheme.  

	We then propose another type of the relaxed DVD method by following the idea of the  IEQ  approach \cite{yang2016linear, zhao2017numerical}. Let us rewrite the free energy functional as
	\begin{equation}
	{\cal E}[\phi]:=\frac {\gamma}2  \left< \phi,\, {\cal L} \phi\right>+\bar{\cal E}_1(\phi)-\frac {|\Omega|} {4\epsilon^2}(2\beta +\beta^2),
	\end{equation}
	where ${\cal L}=-\Delta + \frac\beta{\epsilon^2} $ and $\bar{\cal E}_1(\phi):=\int_\Omega\bar{ E}_1(\phi)\mathbf x=\int_\Omega \frac 1 {4\epsilon^2}(1+\beta-\phi^2)^2\hbox{d}\mathbf x$.
	Here $\beta$ is a suitable stabilization parameter.	
	Let us introduce an auxiliary variable  $Q:=Q(\phi)=\left(\bar E_1(\phi)+C_0\right)^{1/m}$, where $m$ is a positive integer number. 
	For $m=1$, $C_0$ can be  any constant. For an odd integer number $m$, $C_0$ should be a constant such that $\bar E_1(\phi)+C_0\neq 0$.
	For an even integer number $m$, $C_0$ should be a constant such that $\bar E_1(\phi)+C_0>0$, which corresponds to the assumption that the energy $\bar E_1(\phi)$ should be bounded from below.
	With the  auxiliary variable $Q$, one can obtain a modified total free energy 
	\begin{equation}
	\label{modified-free-energy}
	\hat{\cal E}[\phi,Q]=\frac {\gamma}2  \left< \phi,\, {\cal L} \phi\right>+\left<Q^m(\phi)\right>.
	\end{equation}
	For two modified discrete total free energy functionals $\hat{\cal E}^i:=\hat{\cal E}[\phi_i,Q_i]$ and  $\hat{\cal E}^j:=\hat{\cal E}[\phi_j,Q_j]$ with $Q_j$ being the approximation of $Q$ at $t=t_0+c_j h$, we assume that the discrete variational derivative $\mu[\phi_i,\phi_j]$ is a solution of
	\begin{equation}
	\label{d-v-d-ieq}	 
	\hat{\cal E}^i-\hat{\cal E}^j=\left<\phi_i-\phi_j,\,\mu[\phi_i,\phi_j]\right>.
	\end{equation}
	In the left hand side of \eqref{d-v-d-ieq}, we use modified discrete total free energy functional $\hat{\cal E}^i$ to replace the original discrete total free energy functional ${\cal E}^i$. 
	Thus, the \eqref{d-v-d-ieq}	 can be taken as a relaxation of \eqref{d-v-d}. 
	For $\nu=1$, according to  \eqref{modified-free-energy}, we have the discrete variational derivative $\mu[\phi_1,\phi_0]$ taking the following form
	\begin{equation}
	\label{d-v-d-ieq-1}	 
	 \mu[\phi_1,\phi_0]=\frac \gamma 2{\cal L} \left( \phi_1+ \phi_0\right)+\frac{Q^m_1-Q^m_0}{Q_1-Q_0}\frac{\bar E_1'(\bar\phi_{ {1}/{2}})}{m Q^{m-1}(\bar\phi_{ {1}/{2}})},
	\end{equation}
	where $\bar\phi_{ {1}/{2}}$ can be any explicit approximation of $\frac {\phi_1+\phi_0}{2}$ according to the accuracy requirement.
	By taking $Q_1-Q_0=\frac{\bar E_1'(\bar\phi_{ {1}/{2}})}{m Q^{m-1}(\bar\phi_{ {1}/{2}})}(\phi_1-\phi_0)$, one can prove that \eqref{d-v-d-ieq} holds exactly with the discrete variational derivative given by \eqref{d-v-d-ieq-1}.  
	Finally, we can define a new type of relaxed DVD scheme as follows:
	\begin{equation}
	\label{system-4-ieq}
	\left\{
	\begin{aligned}
	 \phi_1&= \phi_0+ h {\cal G} \mu[\phi_1,\phi_0] ,\\
	Q_1&=Q_0+\frac{\phi_1-\phi_0}{m Q^{m-1}(\bar\phi_{ 1/{2}})}\bar E_1'(\bar\phi_{ 1/{2}}).
	\end{aligned}
	\right.
	\end{equation}	
	According to \eqref{d-v-d-ieq}, \eqref{d-v-d-ieq-1}, and \eqref{system-4-ieq} , the unconditional (modified) energy stability of scheme \eqref{system-4-ieq} is obtained by
	\begin{equation}
	\begin{aligned}
	\hat{\cal E}^1-\hat{\cal E}^0&=\left<\phi_1-\phi_0,\,\mu[\phi_1,\phi_0]\right>
	\\
	&=h\left<{\cal G}\mu[\phi_1,\phi_0],\,\mu[\phi_1,\phi_0]\right>\leq 0.
	\end{aligned}
	\end{equation}
	For  $m=2$, the scheme \eqref{system-4-ieq} is the same as the IEQ/CN scheme. 
	
	The relaxed DVD scheme \eqref{system-4-ieq} is  linear and unconditionally stable for general gradient flows, which is constructed by following the idea of the IEQ scheme. 
	In the IEQ scheme, the energy density $\bar E_1(\phi)$ should be bounded from below, which limits the application of the IEQ scheme.
	But, the relaxed DVD scheme \eqref{system-4-ieq} with odd integer $m$, especially for $m=1$, works well for  the unbounded energy density $\bar E_1(\phi)$, which should be an advantage as compared with the IEQ scheme.
	However, the linear system usually involves field variable coefficients, which is  larger than the original linear problem.
	This is the same disadvantage as the IEQ scheme. To overcome it, we propose another relaxed DVD scheme by introducing a scale auxiliary variable similar to the SAV approach \cite{shen2018scalar, shen2019new}.

	Let us introduce a scale auxiliary variable $r:=r(\phi)=\left(\bar{\cal E}_1(\phi)+C_0\right)^{1/m}$.
	Similar to the auxiliary variable $Q$, we recommend using an odd integer number $m$ in the scale auxiliary variable $r$, such as $m=1$.
	A modified total free energy is then defined as	
	\begin{equation}
	\label{modified-free-energy-sav}
	\hat{\cal E}[\phi,r]=\frac {\gamma}2  \left< \phi,\, {\cal L} \phi\right>+r^m.
	\end{equation}
	For two modified discrete total free energy functionals $\hat{\cal E}^i:=\hat{\cal E}[\phi_i,r_i]$ and  $\hat{\cal E}^j:=\hat{\cal E}[\phi_j,r_j]$ with $r_j$ being the approximation of $r$ at $t=t_0+c_j h$, the discrete variational derivative $\mu[\phi_i,\phi_j]$ is selected such that
	\begin{equation}
	\label{d-v-d-sav}	 
	\hat{\cal E}^i-\hat{\cal E}^j=\left<\phi_i-\phi_j,\,\mu[\phi_i,\phi_j]\right>.
	\end{equation}
	The formula \eqref{d-v-d-sav} is another relaxation of  \eqref{d-v-d} since modified discrete total free energy functional $\hat{\cal E}^i$ is used to replace the original discrete total free energy functional ${\cal E}^i$.   
	For $\nu=1$, according to  \eqref{modified-free-energy-sav}, we can derive the discrete variational derivative $\mu[\phi_1,\phi_0]$ as
	\begin{equation}
	\label{d-v-d-sav-1}	 
	 \mu[\phi_1,\phi_0]=\frac \gamma 2 {\cal L}\left( \phi_1+ \phi_0\right)+\frac{r_1^m-r_0^m}{r_1-r_0}\frac{\bar E_1'(\bar\phi_{ 1/{2}})}{m r^{m-1}(\bar\phi_{ 1/{2}})},
	\end{equation}
	where $\bar\phi_{ 1/{2}}$ can be any explicit approximation of $\frac {\phi_1+\phi_0}{2}$ according to the accuracy requirement.
	By taking $r_1-r_0=\frac{\left<\bar E_1'(\bar\phi_{1/{2}})(\phi_1-\phi_0)\right>}{m r^{m-1}(\bar\phi_{ 1/{2}})}$, one can prove that \eqref{d-v-d-sav} holds exactly for the discrete variational derivative given by \eqref{d-v-d-sav-1}.
	
		A new relaxed DVD scheme is then defined as follows:
	\begin{equation}
	\label{system-4-sav}
	\left\{
	\begin{aligned}
	 \phi_1&= \phi_0+ h  {\cal G} \mu[\phi_1,\phi_0],\\
	r_1&=r_0+\frac{\left<\bar E_1'(\bar\phi_{ 1/{2}})(\phi_1-\phi_0)\right>}{m r^{m-1}(\bar\phi_{ 1/{2}})}.
	\end{aligned}
	\right.
	\end{equation}		
	According to \eqref{d-v-d-sav}, \eqref{d-v-d-sav-1} and \eqref{system-4-sav}, we have
	\begin{equation}
	\label{rdvd-sav}
	\begin{aligned}
	\hat{\cal E}^1-\hat{\cal E}^0&=\left<\phi_1-\phi_0,\,\mu[\phi_1,\phi_0]\right>
	\\
	&=h\left<{\cal G}\mu[\phi_1,\phi_0],\,\mu[\phi_1,\phi_0]\right>\leq 0,
	\end{aligned}
	\end{equation}
	which corresponds to the unconditional (modified) energy stability of scheme \eqref{system-4-sav}.	
	For   $m=2$, the scheme \eqref{system-4-sav} is the same as the SAV/CN scheme. 
	
	In summary, we have the following energy stable theorem for all relaxed DVD schemes.
	\begin{theorem}\label{theorem:1-1}
	For any given time step $h >0$,  the relaxed DVD scheme \eqref{d-v-d-cvs-s}, \eqref{system-4-ieq} and \eqref{system-4-sav}
	are unconditionally energy stable and  satisfy the following energy dissipation law
	$$\hat {\cal E}^1\leq \hat{\cal E}^0.$$
	\end{theorem}
	
	The IEQ approach and SAV approach are remarkable as they allow us to construct linear and unconditionally stable schemes for a large class of gradient flows.
	However, those approaches still suffer from the following drawbacks:
	\begin{enumerate}
\item[1.] At each time step, the IEQ approach needs to solve a linear system involving auxiliary variables and the SAV approach needs to solve two linear equations.
\item[2.] For the two approaches, one needs to assume that $\bar{\cal E}_1$ or $\bar E_1$ should be bounded from below. 
\item[3.] For gradient flows with multiple components, the IEQ approach will lead to coupled systems.
\end{enumerate}
	The relaxed DVD method inherits all the advantages of the IEQ and SAV approach but also overcomes most of its shortcomings. 
	If $m$ is an odd integer number, we don't need to assume that $\bar{\cal E}_1$ or $\bar E_1$ should be bounded from below.
	More importantly, for $m=1$, the relaxed DVD schemes \eqref{system-4-ieq} and \eqref{system-4-sav} are fully decoupled schemes.
	At each time step, we only need to solve the first equation of \eqref{system-4-ieq} or \eqref{system-4-sav}, which corresponds to a linear system without auxiliary variables. 
	If necessary, we can explicitly update the auxiliary variables by using the second equation of \eqref{system-4-ieq} or \eqref{system-4-sav}.
	The computational complexity is half  that of the SAV approach and less than half that of the IEQ approach, respectively.

	The relaxed DVD schemes \eqref{system-4-ieq} and \eqref{system-4-sav} both are second-order. 
	However, it is nontrivial to construct a high-order energy stable scheme by using the relaxed DVD method. 
	For $\nu>1$, it is difficulty to find discrete variational derivatives $\mu[\phi_i,\phi_j]$ such that the equation \eqref{d-v-d-ieq} and \eqref{d-v-d-sav} are exactly hold for $1\leq j<i\leq \nu$. 
	We list it as future work. 
	In addition, one can use a class of extrapolated and linearized Runge--Kutta method to build arbitrary high-order energy stable schemes, see \cite{akrivis2019energy} for more details.

	\section{Fully discrete scheme and linear/nonlinear solver}
	\subsection{Fully discrete scheme}
	A fully discrete scheme is then constructed by discretizing the $\nu$ stage DVD scheme \eqref{system-4} or the relaxed DVD schemes  \eqref{d-v-d-cvs-s}, \eqref{system-4-ieq}, and \eqref{system-4-sav} with a spatial discrete method, such as the finite difference method,
	the finite element method, or the spectral method.
	If the relationship \eqref{d-v-d}, \eqref{d-v-d-relax}, \eqref{d-v-d-ieq}, or \eqref{d-v-d-sav} is still exactly held in the fully discrete level and the discretization operator of  ${\cal G}$ is still negative semidefinite, the fully discrete scheme is unconditionally stable and satisfies the energy dissipation law.
	In this paper, we apply the finite difference method to  discrete the $\nu$ stage DVD scheme \eqref{system-4}.	
	The fully discrete scheme for the relaxed DVD schemes  \eqref{d-v-d-cvs-s}, \eqref{system-4-ieq}, and \eqref{system-4-sav} can be constructed similarly.
	
	Without loss of generality, we take the $H^{-1}$ gradient flow as an example to introduce the fully discrete scheme. 
	The operator ${\cal G}=\Delta $ in the $H^{-1}$ gradient flow.
	A one dimensional domain $\Omega=[0,L_x]$  is covered by a uniform mesh with the mesh size $\Delta x = L_x/N_x$. 
	Let us denote  $\tilde{\phi}_{i}^j\approx \phi_i(x_j)$  as the approximate solution $\phi_i$ at $x=x_j$, where
	$x_j= (j-\frac{1}{2})\Delta x$  with $1\leq j \leq N_x$. 
	Throughout the paper, notations with a tilde are the corresponding approximate solutions or functions at the discrete level.
 	Let us introduce some useful discrete operators as follows.
$$[D_x^+\tilde\phi_i]^{j}=\frac{\tilde\phi_i^{j+1}-\tilde\phi_i^{j}}{\Delta x},\quad 
\ [D_x^-\tilde\phi_i]^j=\frac{\tilde\phi^j_{i}-\tilde\phi^{j-1}_{i}}{\Delta x},\quad \
$$
 $$
[D_x \tilde\phi_i]^j=\frac{\tilde\phi_i^{j+\frac{1}{2}}-\tilde\phi_i^{j-\frac{1}{2}}}{\Delta x}, 
\quad \ 
 \left[(D_x^\pm\tilde\phi_i)^2\right]^j=\frac{\left([D_x^+\tilde\phi_i]^j\right)^2+\left([D_x^-\tilde\phi_i]^j\right)^2}{2}.$$
	The operator $\cal G$ is discretized by
$[\Delta\tilde\phi_i]^j =  [D_x^2 \tilde\phi_i]^j$.
With the periodic boundary conditions or the homogeneous Neumann boundary conditions, 
we present  vital formula
\begin{equation}
\label{summation}
\begin{aligned}
\sum_{j=1}^{N_x}\tilde \varphi^j\left[D_x^2\tilde \phi_i\right]^{j}\color{black}
=-\frac 1 2
\sum_{j=1}^{N_x}
\left([D_x^+ \tilde \varphi]^{j} [D_x^+\tilde \phi_i]^{j}\right),
   \end{aligned}
\end{equation}
\color{black}
which implies that the discretization operator of $\cal G$ is negative semidefinite.

	We then define the  discrete free energy functional as 
	\begin{equation}
	\label{free-energy-dis}
	\tilde{\cal E}^i=\sum\limits_{j=1}^{N_x}\Delta x\left\{\frac{\gamma}{2}\left[(D_x^\pm\tilde{\phi}_i)^2\right]^j+E_1[\tilde\phi^j_i]\right\}.
	\end{equation}
	Correspondingly, we can derive the discrete variational derivative $\tilde\mu[\tilde \phi_i^j,\tilde\phi_s^j]$ as
	\begin{equation}
	 \tilde y^j_{\frac {s(2\nu-s+1)}{2}+(i-s)}:=\tilde\mu[\tilde \phi_i^j,\tilde\phi_s^j]=-\frac \gamma 2 \left([D_x^2 \tilde\phi_i]^j+[D_x^2 \tilde\phi_s]^j\right) +E_1\{\tilde\phi^j_i,\tilde\phi_s^j\},
	\end{equation}
	where $i>s$.
	With the definition, we have
	\begin{equation}
	\label{d-v-d-f}
	 \tilde{\cal E}^i-\tilde{\cal E}^s=\sum\limits_{j=1}^{N_x}\Delta x(\tilde\phi^j_i-\tilde\phi^j_s)\tilde\mu[\tilde \phi_i^j,\tilde\phi_s^j],
	\end{equation}
	which implies that the formula \eqref{d-v-d} holds in the discrete level. 
	Finally, the fully discrete scheme is given by 
	\begin{equation}
	\label{fully-system}
	\begin{aligned}
	 \tilde\phi_i^j= \tilde\phi^j_0+ h \sum\limits_{s=1}^{\hat\nu}\tilde{ a}_{is} [D_x^2 \tilde y_s]^j \qquad (i=1,\,2,\,\cdots,\, \nu).
	\end{aligned}
	\end{equation}
	\begin{theorem}\label{theorem:2}
	For any given time step $h >0$,  the fully discrete schemes \eqref{fully-system} with $\tilde{ a}_{is} $ given by \eqref{Sch-1}-\eqref {Sch-4}
	are unconditionally energy stable and  satisfy the following energy dissipation law
	$$\tilde{\cal E}^\nu\leq \tilde{\cal E}^0.$$
	\end{theorem}
	The proof for \cref{theorem:2} is similar to the proof for \cref{theorem:1} and will not be given in this paper.
	
	\subsection{Linear/nonlinear solver}
	By discretizing the gradient flow \eqref{system-1} with the proposed energy stable scheme \eqref{fully-system}, 
	a discrete nonlinear equation system is  
	constructed and required to be solved at each time step. 
	This should be a   disadvantage compared to linear schemes such as the IEQ and SAV methods since nonlinear algebraic systems are more expensive to solve than linear algebraic systems. 
	In this paper, we solve the nonlinear system by using the inexact Netwon type algorithm \cite{NKS}.
	At each time step, the nonlinear algebraic system is denoted as ${\cal F}(\mathbf{X})=0$, and the unknown $\mathbf{X}$ is organized by points, i.e.  
$\mathbf{X}=(\tilde \phi^1_{1},\, \tilde \phi^1_{2}$, $\cdots,\, \tilde{\phi}^1_{\nu},\,\tilde \phi^2_{1},\, \tilde \phi^2_{2} ,\,\cdots)^{T}$.
	Assume $\mathbf{X}_{m}$ is the solution at the $m$-th Newton iteration, then a new solution $\mathbf{X}_{m+1}$ can be obtained by the following steps.
 \begin{itemize}
 \item Compute {an} inexact Newton direction $\mathbf{S}_m$ by solving the following Jacobian system with an iterative method:
 \begin{equation}\label{eq:line-system}
\begin{array}{ll}
J_m\mathbf{S}_m=-{\cal F}(\mathbf{X}_m),
\end{array}
\end{equation}
where  $J_m:=\nabla {\cal F}(\mathbf{X}_m)$ is the Jacobian matrix at the solution $\mathbf{X}_m$. 

 \item Calculate a new solution through $\mathbf{X}_{m+1}=\mathbf{X}_{m} +\lambda_m\mathbf{S}_m$, where the step length $\lambda_{m}\in (0,1]$ is determined  by a line search procedure \cite{NMFU}.
\end{itemize}
The   stopping condition for the nonlinear iteration is set as
\begin{equation}
\|{\cal F}(\mathbf{X}_{m+1})\| \leq \max\{\varepsilon_r
\|{\cal F}(\mathbf{X}_0)\|, \varepsilon_a\},
\end{equation}
where $\varepsilon_r=10^{-12}, \varepsilon_a =10^{-12}$ are the relative and absolute tolerances for the 
nonlinear iteration, respectively.

In each time step of the relaxed DVD schemes or each iterative step of the Newton solver,
we need to solve a large sparse linear system.
Let us denote the linear system as ${\cal A}{\cal X}=\mathbf{b}$.
In this paper, we solve the linear system by using a restarted Generalized Minimal Residual (GMRES) method \cite{gmres} with right-preconditioner until the linear residual $\mathbf{R} = {\cal A}{\cal X}-\mathbf{b}$ satisfies 
the stopping condition as follows
$$\|\mathbf{R}\| \leq  \textnormal{max} \{\xi_r\|\mathbf b\|, \xi_a\}.$$ 
Here $\xi_r, \xi_a $ are the relative and absolute tolerances for the 
linear iteration, respectively. 
For the relaxed DVD schemes, we take the tolerances as $\xi_r=10^{-12}, \xi_a=10^{-12} $.
In each iterative step of the Newton solver, the tolerances are set as $\xi_r=10^{-14}, \xi_a=10^{-3} $.
The preconditioned linear system is written as
\begin{equation}
\label{hjacobi}
{\cal A}H^{-1}(H{\cal X}) =\mathbf b,
\end{equation}
where $H^{-1}$ is the additive Schwarz type preconditioner. 
In this paper, we use the standard additive Schwarz preconditioner \cite{DDA}.
More comparison of the standard additive Schwarz preconditioner, 
the left restricted additive Schwarz \cite{cai99sisc} preconditioner, 
and the right restricted  additive Schwarz \cite{cai03sinum_ash} preconditioner can be found in \cite{huang2020parallel, huang2022parallel, wei2019parallel}.

\section{Numerical simulations}
In this section, we apply the DVD scheme \eqref{system-4}, the relaxed DVD schemes \eqref{system-4-ieq} and \eqref{system-4-sav} to solve the Allen--Cahn (AC) equation and Cahn--Hilliard (CH) equation. 
We compare four DVD schemes and two relaxed DVD schemes, whose abbreviations are given in \cref{abbreviation}.
We mainly focus on the efficiency, stability, and accuracy of the newly proposed methods.
The algorithms for the phase field system are implemented on top of the Portable, Extensible Toolkits for Scientific computations (PETSc) library \cite{balay2019petsc}.
The numerical experiments are performed on the LSSC-IV supercomputer. 
\begin{table}[!htb]
\caption{
The abbreviations and comparison for six DVD schemes.
}
\label{abbreviation} 
\centering
\begin{tabular}{|c|cccc||cc|}
\hline
&\multicolumn{4}{|c||}{The DVD scheme \eqref{system-4}}   &\multicolumn{2}{|c|}{The relaxed DVD scheme \eqref{system-4-sav}} \\
\hline
&\multicolumn{4}{|c||}{Implicit}   &\multicolumn{2}{|c|}{Semi-implicit} \\
\hline
Abbr & Sch-1&Sch-2 & Sch-3& Sch-4& R-DVD-1& SAV/CN\\
\hline
$\mathbb{A}$ & \eqref{Sch-1}& \eqref{Sch-2} &  \eqref{Sch-3}&  \eqref{Sch-4}&$-$& $-$\\
\hline
$m$ &$-$& $-$ &  $-$&  $-$&1& 2\\
\hline
order &2& 3 &  4&  4&2& 2\\
\hline
\end{tabular}
\end{table}

\subsection{Cahn--Hilliard equation}
{\it Example 1.} (Convergence rate of the newly proposed schemes for the standard CH equation.)
The computational domain is set as $(0,\,2\pi)^2$, which is covered by a $500\times500$ uniform mesh. 
The parameters $\epsilon = 0.1$ and $\gamma=1$. 
The initial data is chosen as: $\phi(\mathbf{x},0)=0.05 \sin x_1 \sin x_2$. 
The stabilization parameter $\beta$ is chosen as 1.
In the relaxed DVD scheme \eqref{system-4-sav}, $\bar\phi_{n+1/2}=\frac 3 2 \phi_n-\frac 1 2\phi_{n-1}$.
The numerical solutions at $t=0.01$ obtained by using $\Delta t$ sufficiently small are taken as the reference solutions.
The numerical errors at $t=0.01$ for six schemes are shown in \cref{convergencerate}, where we can observe the idea convergence rate for all schemes. 
While the Sch-1, R-DVD-1, and SAV/CN schemes are second-order accurate, the numerical error of the implicit scheme ( Sch-1) is an order of magnitude smaller than that of the other two semi-implicit schemes (R-DVD-1, and SAV/CN). 
We further compare the efficiency of the three second-order schemes. 
The time step size is set as 1e-4 for the implicit scheme (Sch-1) and 2.5e-5 for the other two second-order semi-implicit schemes (R-DVD-1, and SAV/CN),
such that the numerical errors coming from the three second-order schemes are almost the same. 
The total computing times and the linear/nonlinear iteration numbers are shown in \cref{efficiency}. 
It can be observed that the computation time per time step of the R-DVD-1 scheme is almost half that of the other two schemes due to the lowest computational complexity per time step.
However, due to the accuracy advantage, the implicit scheme (Sch-1) outperforms other second-order schemes in the total computing time.

\begin{table}[h]
\caption{
Errors and convergence rates of  the DVD schemes and the relaxed DVD schemes  for
the CH equation.
}
\label{convergencerate} 
\centering
\begin{tabular}{|c|c|cccc|}
\hline
{Scheme} &$h$&1e-4 &5e-5 &2.5e-5 &1.25e-5 \\
\hline
  \multirow{2}{*}{Sch-1}     		&Error                                &1.95e-4                  &4.79e-5                     &1.08e-5 &1.78e-6 \\
                                    			&     Rate          &-                          &2.02              &2.15 &2.59 \\
\hline
  \multirow{2}{*}{R-DVD-1}     		&Error                                &3.42e-3                  &8.71e-4                     &2.21e-4 &5.34e-5 \\
                                    			&     Rate          &-                          &1.97              &1.98 &2.05 \\
\hline
  \multirow{2}{*}{SAV/CN}     		&Error                                &3.43e-3                  &8.75e-4                     &2.22e-4 &5.36e-5 \\
                                    			&     Rate          &-                          &1.97              &1.98 &2.05 \\
                                    \hline
                                    \hline
                               Scheme     &$h$&5e-4 &2.5e-4 &1.25e-4 &6.25e-5 \\
                                    \hline
  \multirow{2}{*}{Sch-2}     		&Error             &2.09e-4                   &1.48e-5                  &1.75e-6                 &2.18e-7 \\
                                    		&     Rate          &-                                &3.82             &3.09&3.00 \\
                                    \hline
  \multirow{2}{*}{Sch-3}     		&Error             &2.10e-4                   &1.24e-5                  &7.72e-7                 &4.88e-8 \\
                                    		&     Rate          &-                                &4.08             &4.01&3.98 \\
                                    \hline
  \multirow{2}{*}{Sch-4}     		&Error             &9.05e-5                   &5.43e-6                 &3.34e-7                 &1.98e-8 \\
                                    		&     Rate          &-                                &4.06             &4.01&4.08 \\
\hline
\end{tabular}
\end{table}

\begin{table}[H]
\caption{
The total linear/nonlinear iterations and total computing time of  the Sch-1, R-DVD-1, and SAV/CN  for
the CH equation. All simulations are run by using 144 processors. 
}
\label{efficiency} 
\centering
\begin{tabular}{|c|c|cccc|}
\hline
{Scheme} &$h$&Time steps  & Newton &GMRES &Computing time (s)\\
\hline
{Sch-1}     		&1e-4                              &100              &341                    &5,697 & 18.07 \\
\hline
\hline
{R-DVD-1}       		&2.5e-5                              &400              &$-$                    &12,374&33.87 \\
\hline
\hline
SAV/CN  		&2.5e-5                              &400              &$-$                    & 25,780&72.38 \\
\hline
\end{tabular}
\end{table}

\begin{figure}[H]
\hspace{0.0\textwidth} \subfloat[]{
\includegraphics[width=0.48\textwidth]{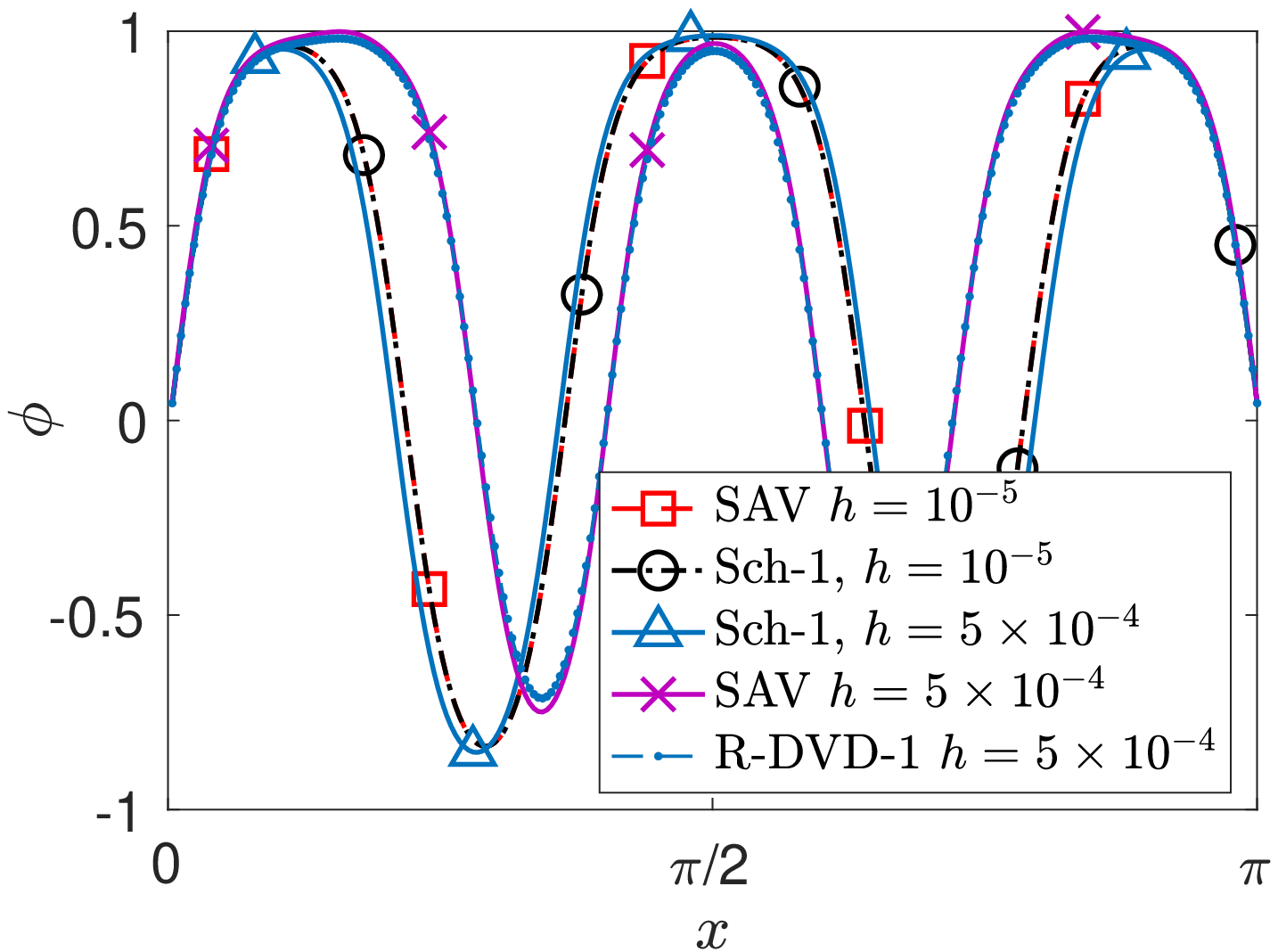}}
\hspace{0.0\textwidth} \subfloat[]{
\includegraphics[width=0.48\textwidth]{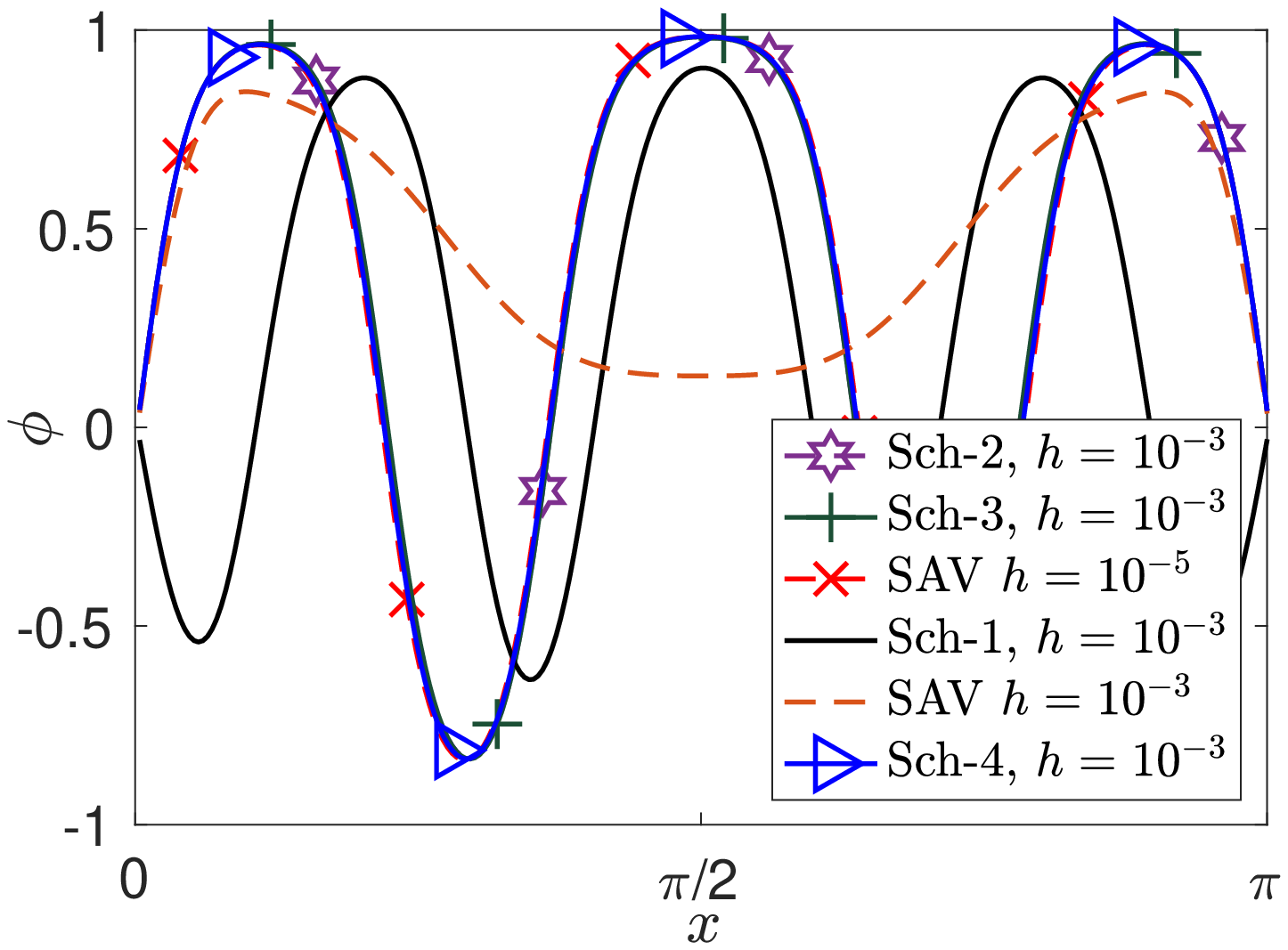}}
\hspace{0.0\textwidth} \subfloat[]{
\includegraphics[width=0.48\textwidth]{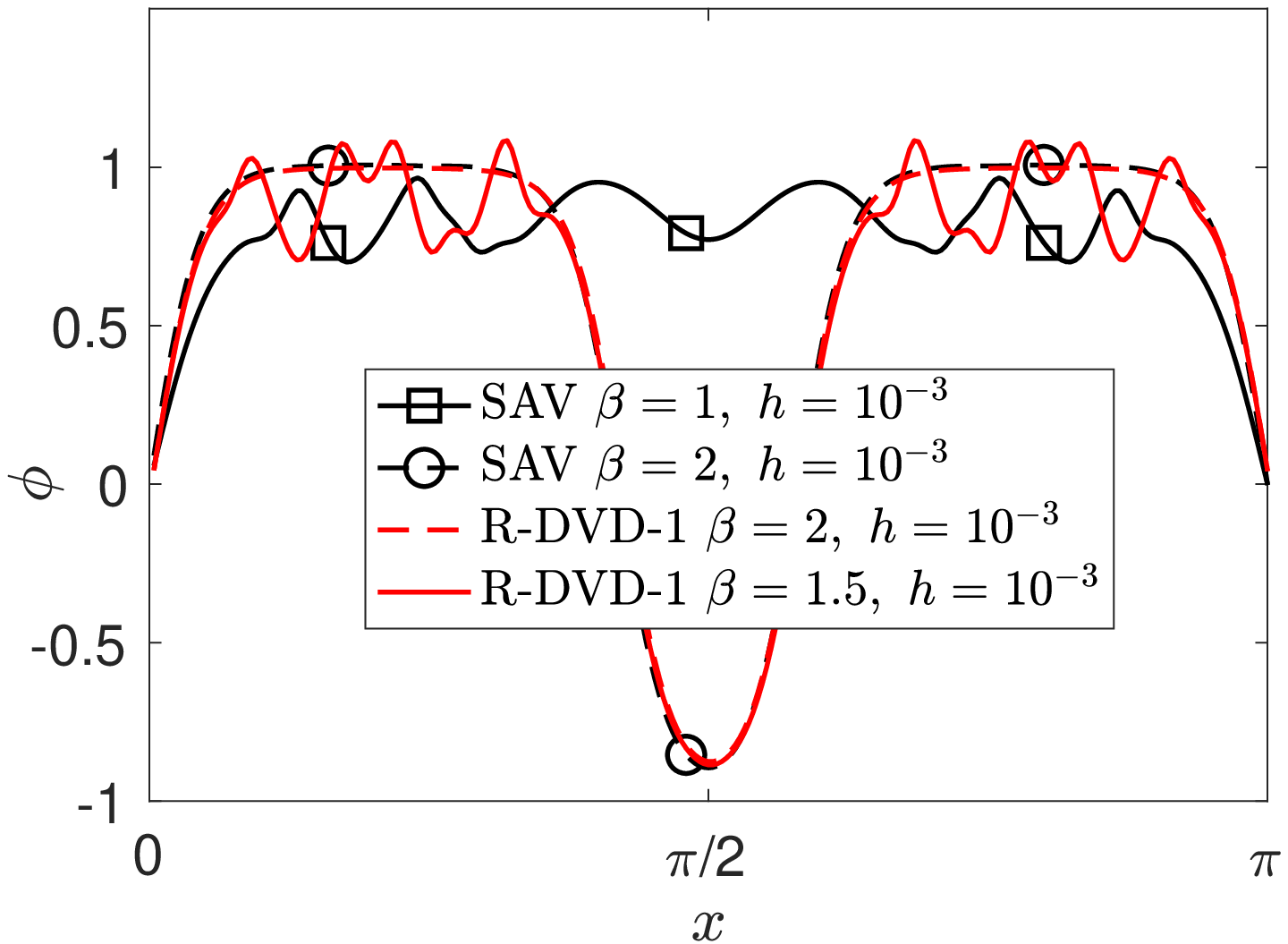}}
\hspace{0.0\textwidth} \subfloat[]{
\includegraphics[width=0.48\textwidth]{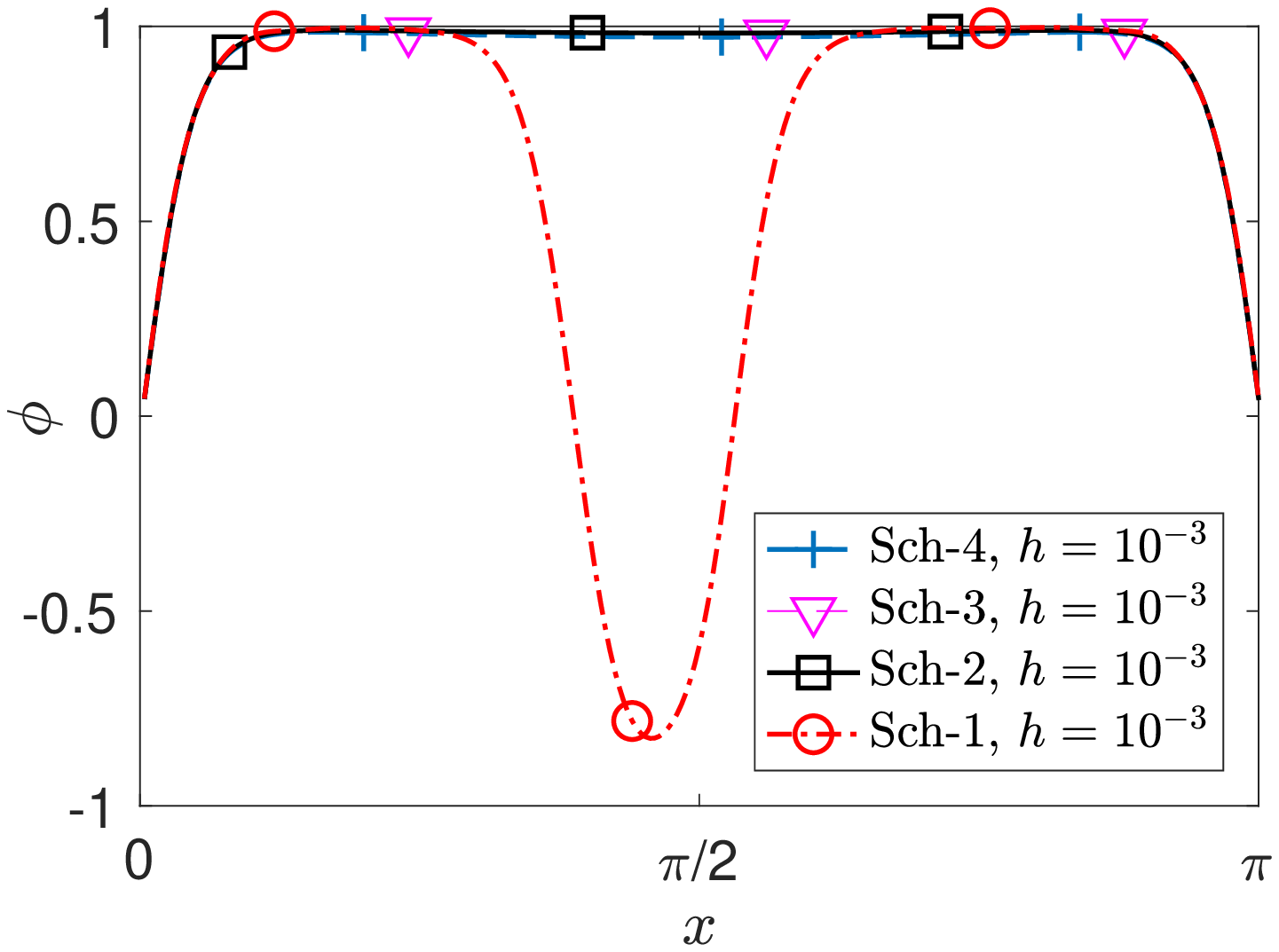}}
\caption{ Numerical solutions of $\phi$ at different time. (a)  $T = 0.01$.  $h = 5\times 10^{-4}$. (b) $T = 0.01$.  $h =  10^{-3}$.
(c)  $T = 0.1$.  $h= 10^{-3}$. (d)  $T = 0.1$.  $h=  10^{-3}$.
}
\label{stable}
\end{figure}
{\it Example 2}. (Stability verification of the newly proposed schemes for the standard CH equation.)
The computational domain is set as $\Omega=(0,\,2\pi)$, which is covered by a  uniform mesh with mesh size being $\pi/250$. 
The parameter $\epsilon = 0.1$,  and $\gamma=1$. 
In the relaxed DVD scheme \eqref{system-4-sav}, $\bar\phi_{n+1/2}=\frac 3 2 \phi_n-\frac 1 2\phi_{n-1}$.
The initial condition is $\phi(x,\,0)=0.2\sin x$. 
The numerical solutions at $T=0.01$ are plotted in \cref{stable} (a)-(b). 
The implicit scheme Sch-1 can get an accurate solution with time step size $h=5\times10^{-4}$, while the other second-order semi-implicit schemes have less accurate solutions.
With a large time step size $h=10^{-3}$, the high-order DVD scheme (Sch-2, Sch-3, Sch-4) can get accurate solutions, while the second-order implicit/semi-implicit schemes are inaccurate.
The numerical solutions at $T=0.1$ are plotted in \cref{stable} (c)-(d). 
For a large time step size $h=10^{-3}$, we find that the four implicit schemes (Sch-1, ..., Sch-4) are unconditionally stable and the semi-implicit schemes (R-DVD-1 and SAV) are stable with a suitable parameter $\beta$.
The evolutions of the modified total free energy for the semi-implicit schemes (R-DVD-1 and SAV) are given in \cref{energy}-(a), which validates the unconditional (modified) energy stability of the semi-implicit schemes.
However,  the semi-implicit schemes do not satisfy the (original) energy dissipation law, as shown in \cref{energy}-(b).
The dynamic processes of the high-order implicit schemes (Sch-2, Sch-3, Sch-4) agree well with the reference dynamic process obtained by the SAV/CN scheme with a small time step size $h=0.00002$.
In contrast, the results obtained by the second-order schemes are visibly different from the reference dynamic process.

\begin{figure}[h]
\hspace{0.0\textwidth} \subfloat[]{
\includegraphics[width=0.48\textwidth]{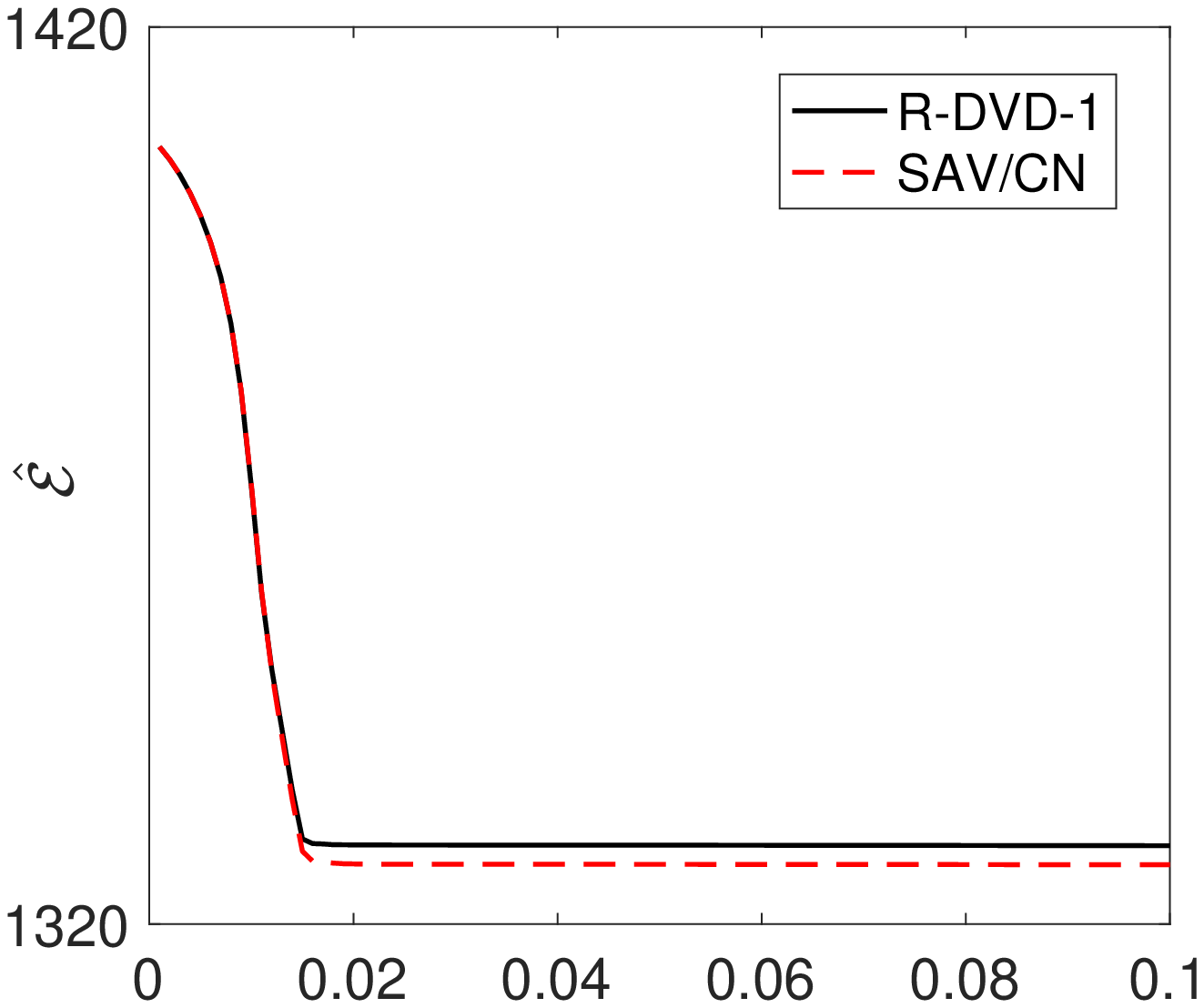}}
\hspace{0.0\textwidth} \subfloat[]{
\includegraphics[width=0.48\textwidth]{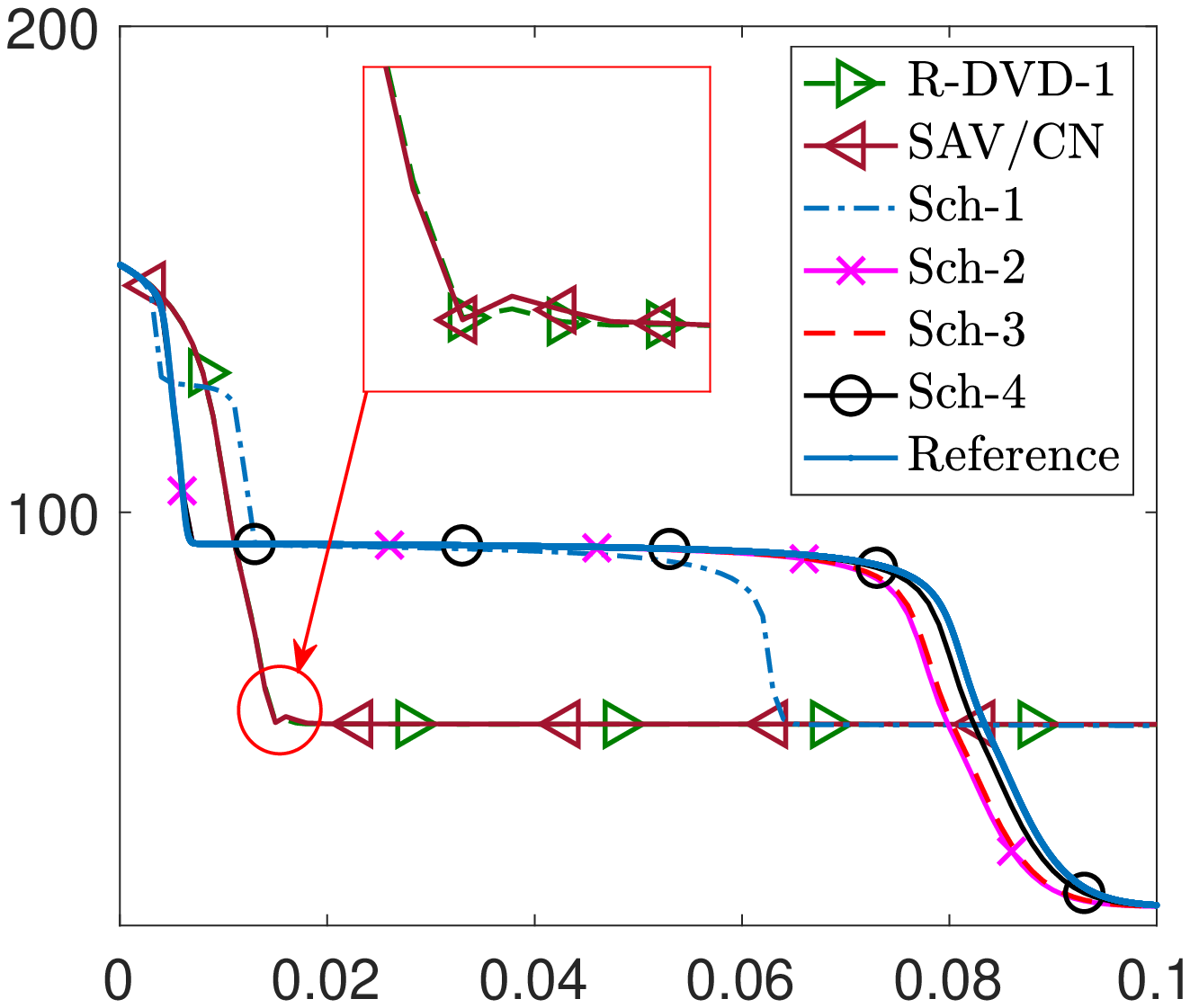}}
\caption{ Time step size $h=0.001$ and $\beta=2$. Here reference solution is obtained by the SAV/CN scheme with a small time step size $h=0.00002$. (a) Modified total free energy $\hat{\cal E}$. (b) Original total free energy ${\cal E}$. 
}
\label{energy}
\end{figure}

\subsection{Allen--Cahn equation}

{\it Example 3}. This example is a benchmark problem for the AC equation (see \cite{chen1998applications}).
The computational domain is chosen as $\Omega=(-128,\,128)^2$, which is covered by a $1024\times 1024$ uniform mesh. After mapping the computational domain to $(-1,1)^2$, the mobility parameter $\gamma=6.10351 \times 10^{-5}$ 
and $\epsilon = 0.0078$. Initially, a circle of radius $R_0=100$ is located at the center of the computational domain. The initial condition is given by
\begin{equation}
		\phi(x_1,x_2,0) = \left\{
	\begin{aligned} &1,\quad x_1^2+x_2^2< R_0^2,\\
	&-1,\quad x_1^2+x_2^2\geq R_0^2.
	\end{aligned}
	\right.
	\end{equation}
	
	In the sharp interface limit ($\epsilon \rightarrow 0$, which is suitable for $\epsilon = 0.0078$), the radius of the circle at  time $t$ is given by
	\begin{equation}
		R(t)=\sqrt{R_0^2 -{2t}}.
	\end{equation}
	The computed radius $R(t)$ obtained by using the linear schemes (R-DVD-1 or SAV/CN) is given in \cref{radius}-(a), which keeps monotonically decreasing and is very close to the sharp interface limit value.
	Correspondingly, the computed radius $R(t)$ get by using the nonlinear schemes (Sch-1, ..., Sch-4) is given in \cref{radius}-(b), which shows that the high-order schemes can obtain similar results by using a large time step size, such as $h=4$.
	
	\begin{figure}[H]
\hspace{0.0\textwidth} \subfloat[]{
\includegraphics[width=0.48\textwidth]{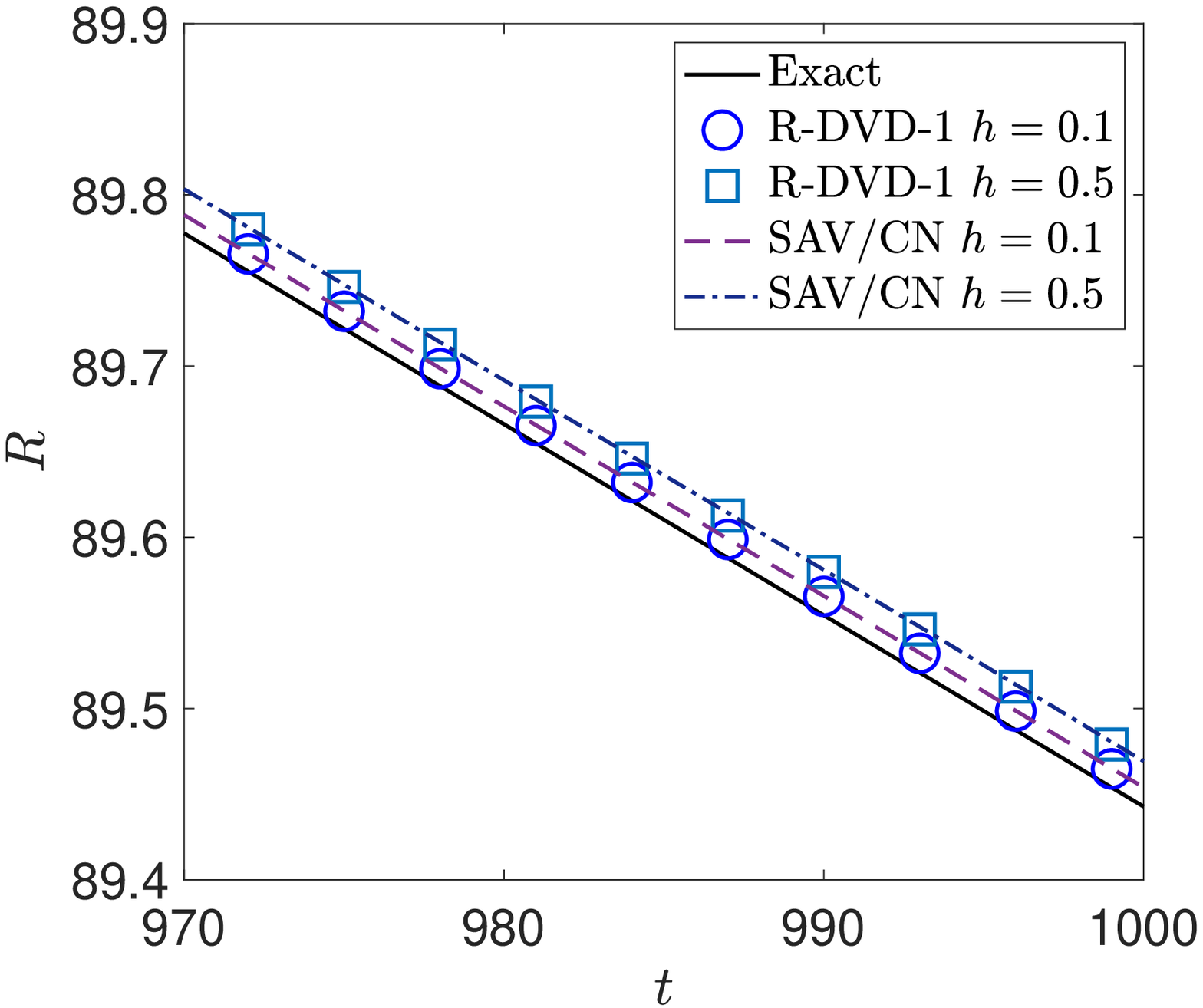}}
\hspace{0.0\textwidth} \subfloat[]{
\includegraphics[width=0.48\textwidth]{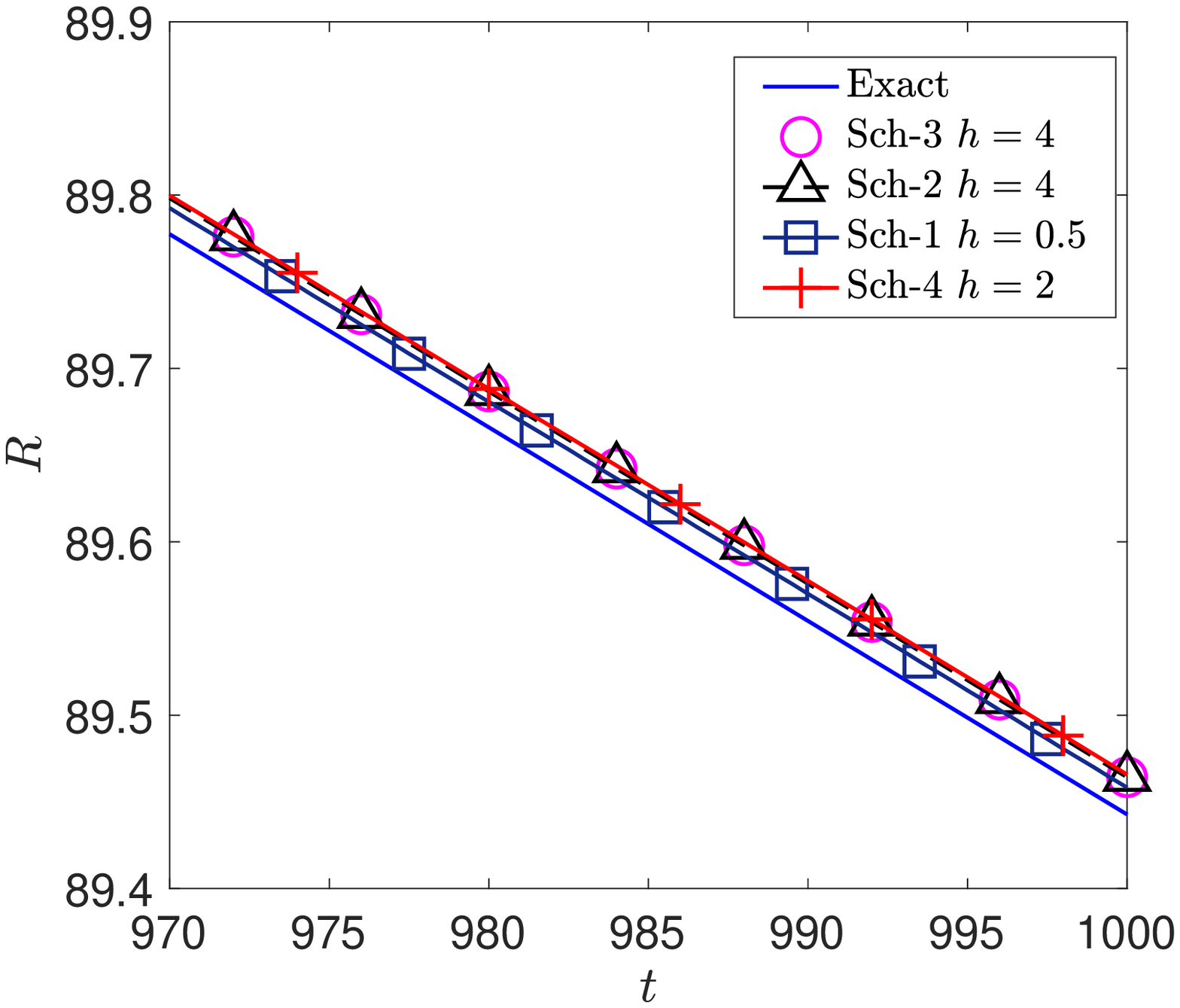}}
\caption{ The evolution of radius $R(t)$. 
}\label{radius}
\end{figure}

{\it Example 4}.
The computational domain is set as $(0,\,2\pi)$, which is covered by a uniform mesh with mesh size being $\pi/250$. 
The parameters in AC equation are taken as $\epsilon = 0.1$, and $\gamma=1$. 
The initial data is set to be a randomly distributed state
$\phi(\mathbf{x},0)=\delta _\phi$, where $\delta _\phi$  is a uniform random distribution function of $−0.02$ to $0.02$. 
The numerical solutions obtained by using SAV/CN with a sufficiently small  $h=0.005$ are taken as the reference solutions.
We plot the numerical solutions at $t=0.1$ and $t=1$ using the newly proposed schemes with a large time step size $h=0.01$ in \cref{randomresults}. 
The numerical results of all schemes agree well with the reference solution at $t=1$, but the solutions obtained by the semi-implicit schemes (R-DVD-1 and SAV/CN) have visible differences with the reference solution at $t=0.1$.
This example indicates that implicit schemes, especially high-order implicit schemes, are more accurate.

\section{Conclusion}

In this paper, we proposed a framework to construct an arbitrary high-order DVD scheme or relaxed DVD scheme for dealing with a large class of gradient flows.
By combining with the Runge--Kutta process, the arbitrary high-order DVD scheme is built by selecting a suitable definition for the discrete variational derivative.
A sufficient condition is introduced to verify the unconditional energy stability of the newly proposed scheme. 
We prove that the newly proposed arbitrary high-order DVD scheme is unconditionally stable to the original total free energy without any additional assumptions.
The newly proposed arbitrary high-order DVD scheme leads to a nonlinear algebraic system when fully discretized, which is more complex and less efficient than a linear scheme.
To overcome this difficulty, several relaxed DVD schemes are constructed by combining ideas from existing linear approaches, such as IEQ and SAV. 
The relaxed DVD scheme is unconditionally stable to a modified total free energy.
For linear schemes, we recommend readers to use the R-DVD-1 scheme, which is a second-order scheme and is more efficient than the IEQ/CN and SAV/CN scheme.
	\begin{figure}[H]
\hspace{0.0\textwidth} \subfloat[R-DVD-1, $t=0.1$]{
\includegraphics[width=0.305\textwidth]{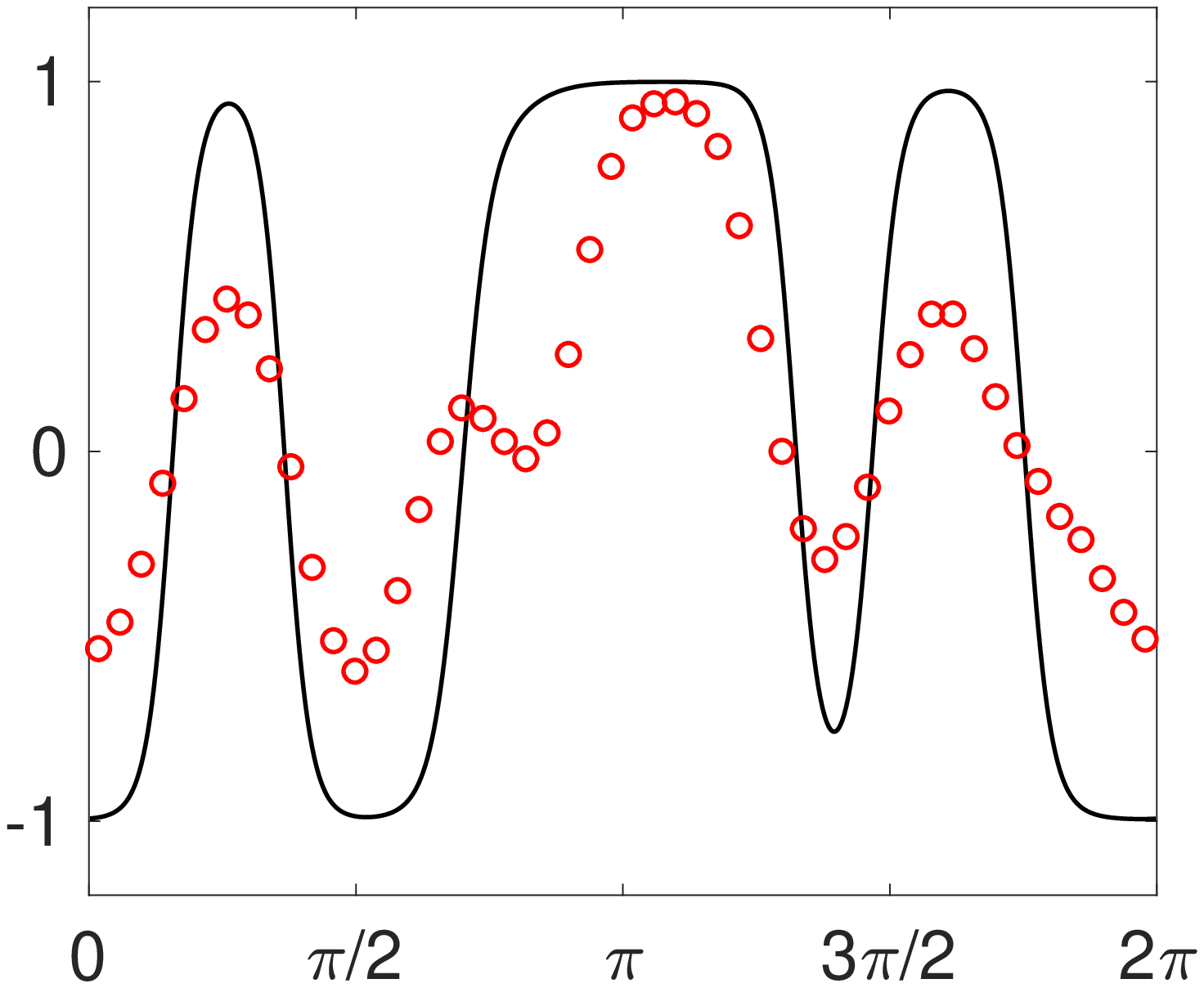}}
\hspace{0.0\textwidth} \subfloat[SAV/CN, $t=0.1$]{
\includegraphics[width=0.305\textwidth]{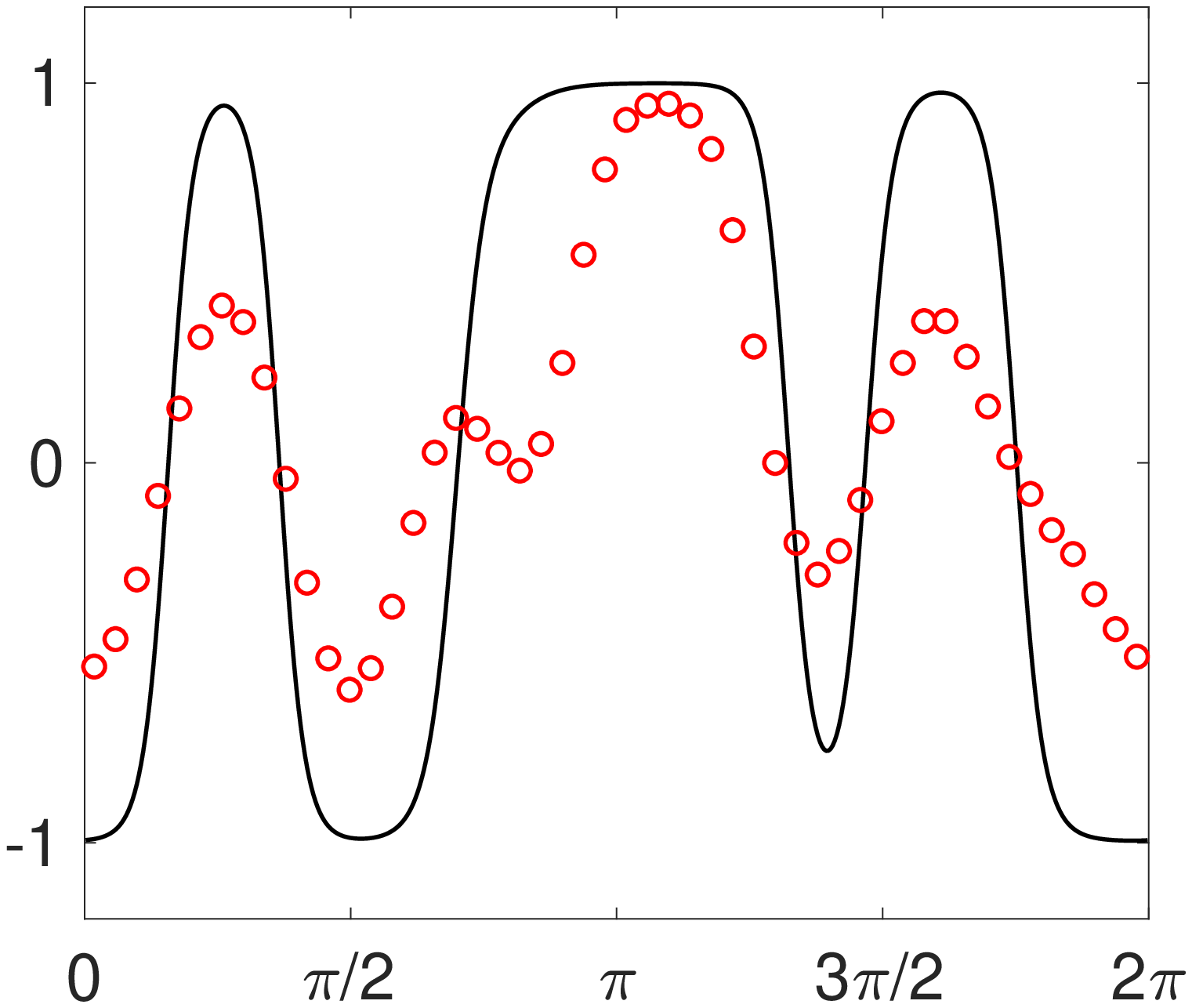}}
\hspace{0.0\textwidth} \subfloat[Sch-1, $t=0.1$]{
\includegraphics[width=0.305\textwidth]{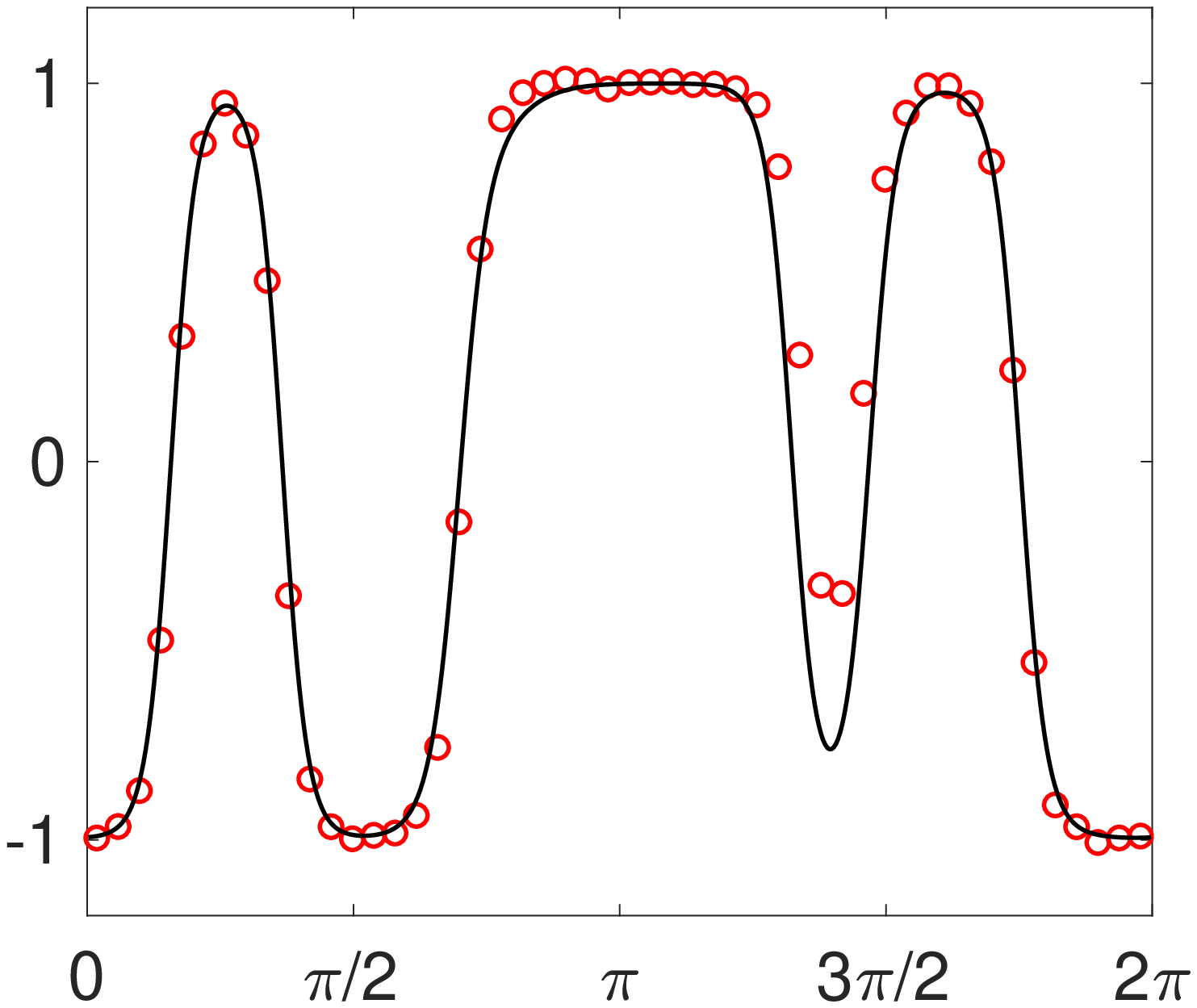}}
\hspace{0.0\textwidth} \subfloat[Sch-2, $t=0.1$]{
\includegraphics[width=0.305\textwidth]{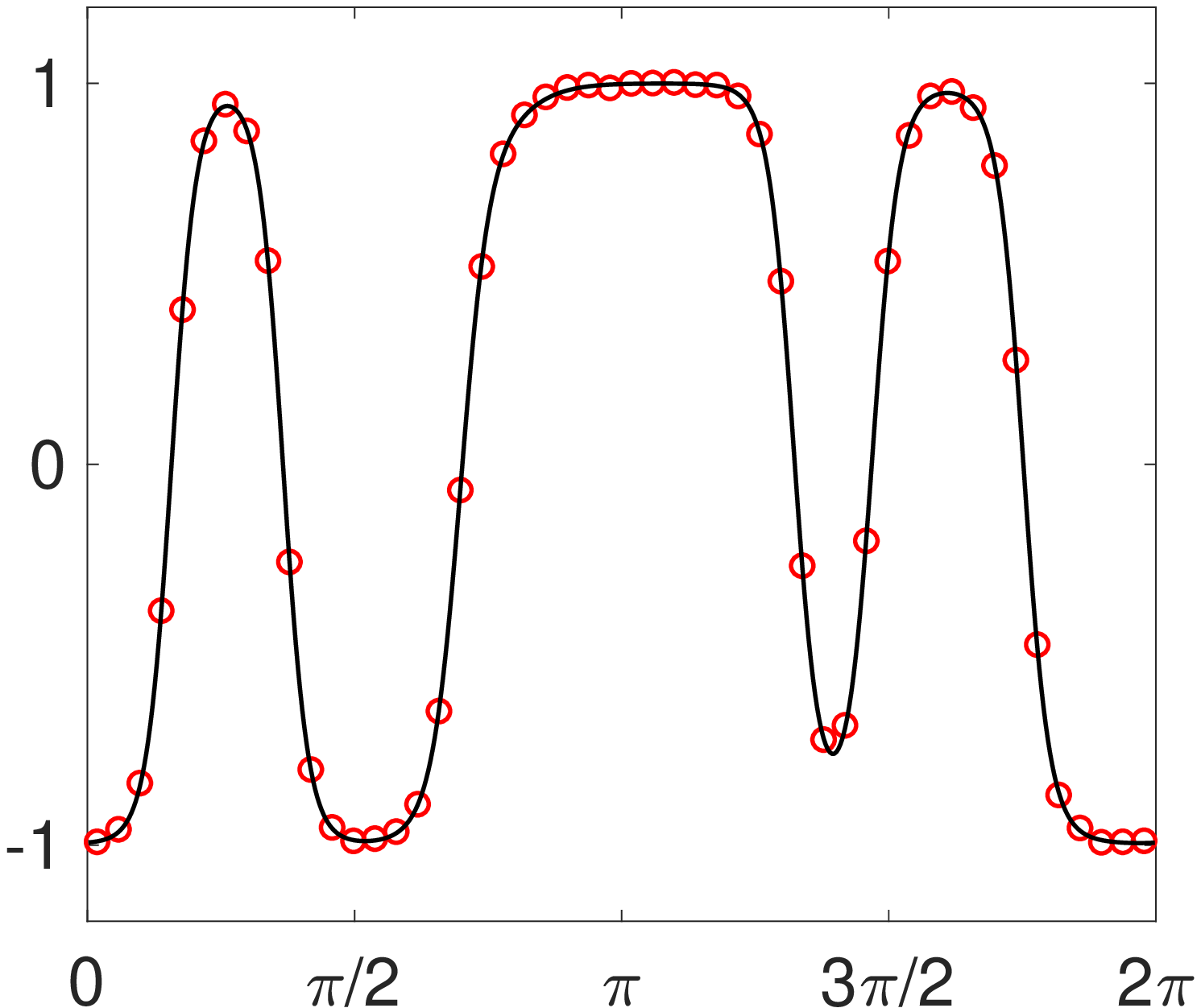}}
\hspace{0.0\textwidth} \subfloat[Sch-3, $t=0.1$]{
\includegraphics[width=0.305\textwidth]{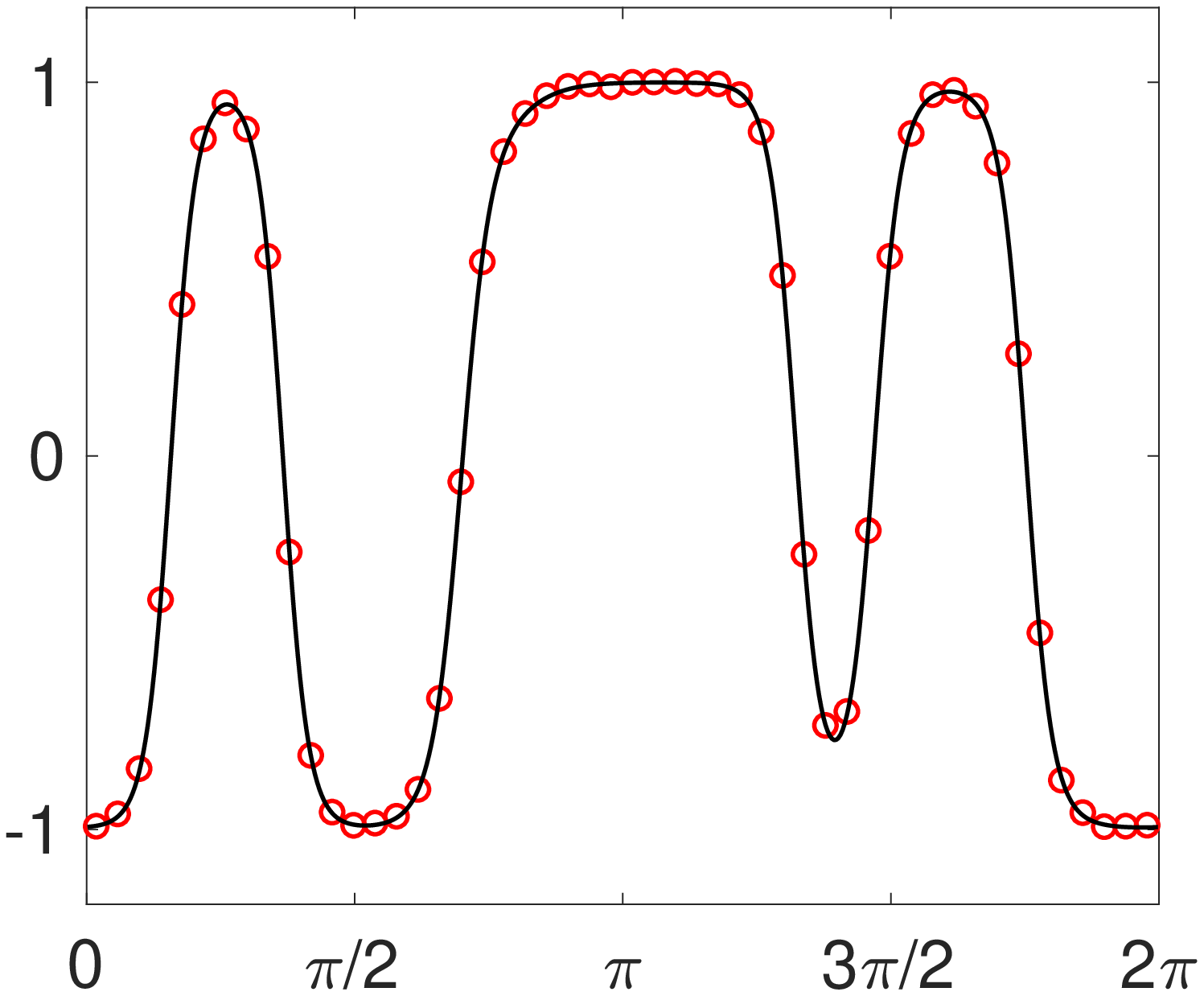}}
\hspace{0.0\textwidth} \subfloat[Sch-4, $t=0.1$]{
\includegraphics[width=0.305\textwidth]{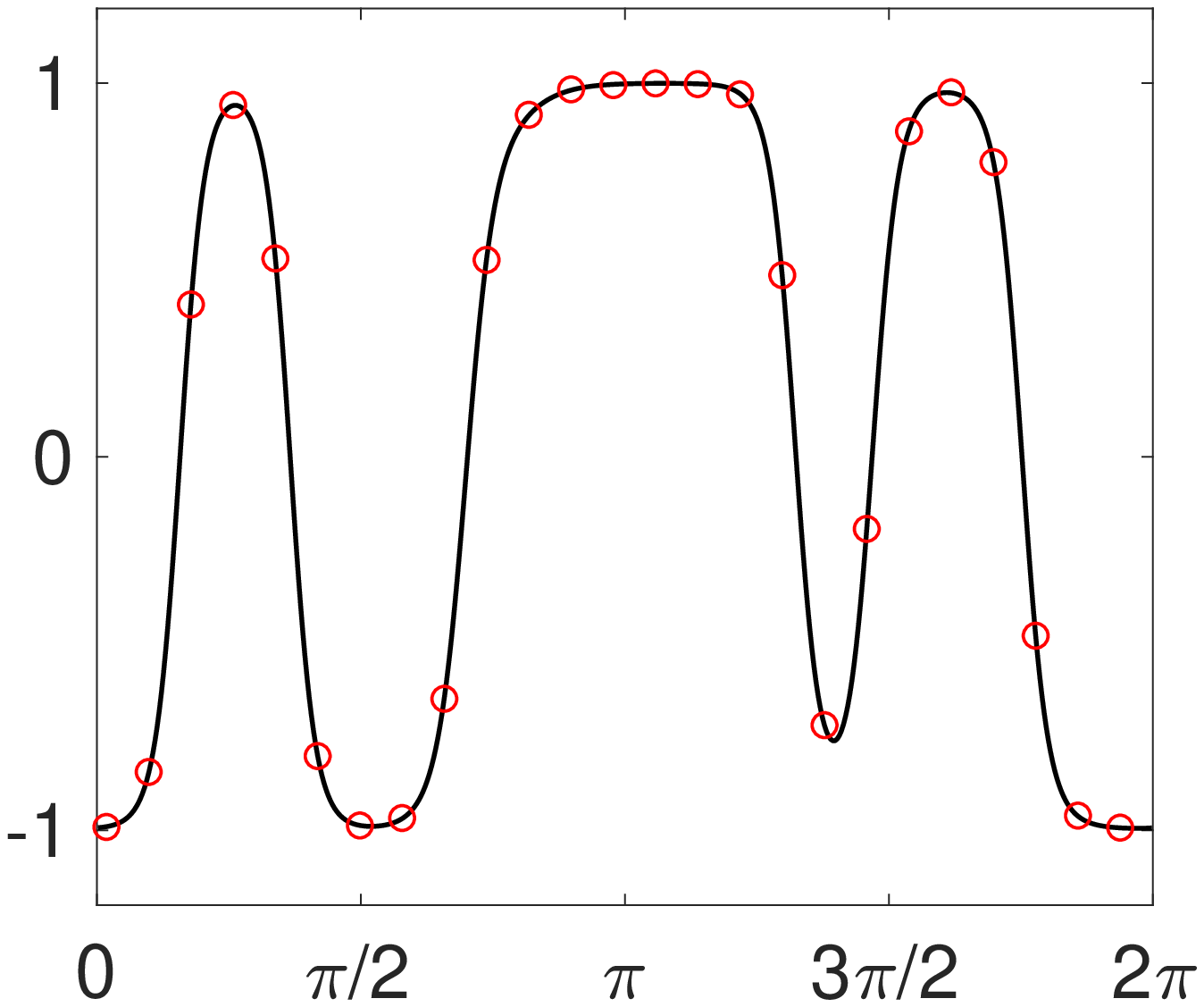}}
\hspace{0.0\textwidth} \subfloat[R-DVD-1, $t=1$]{
\includegraphics[width=0.305\textwidth]{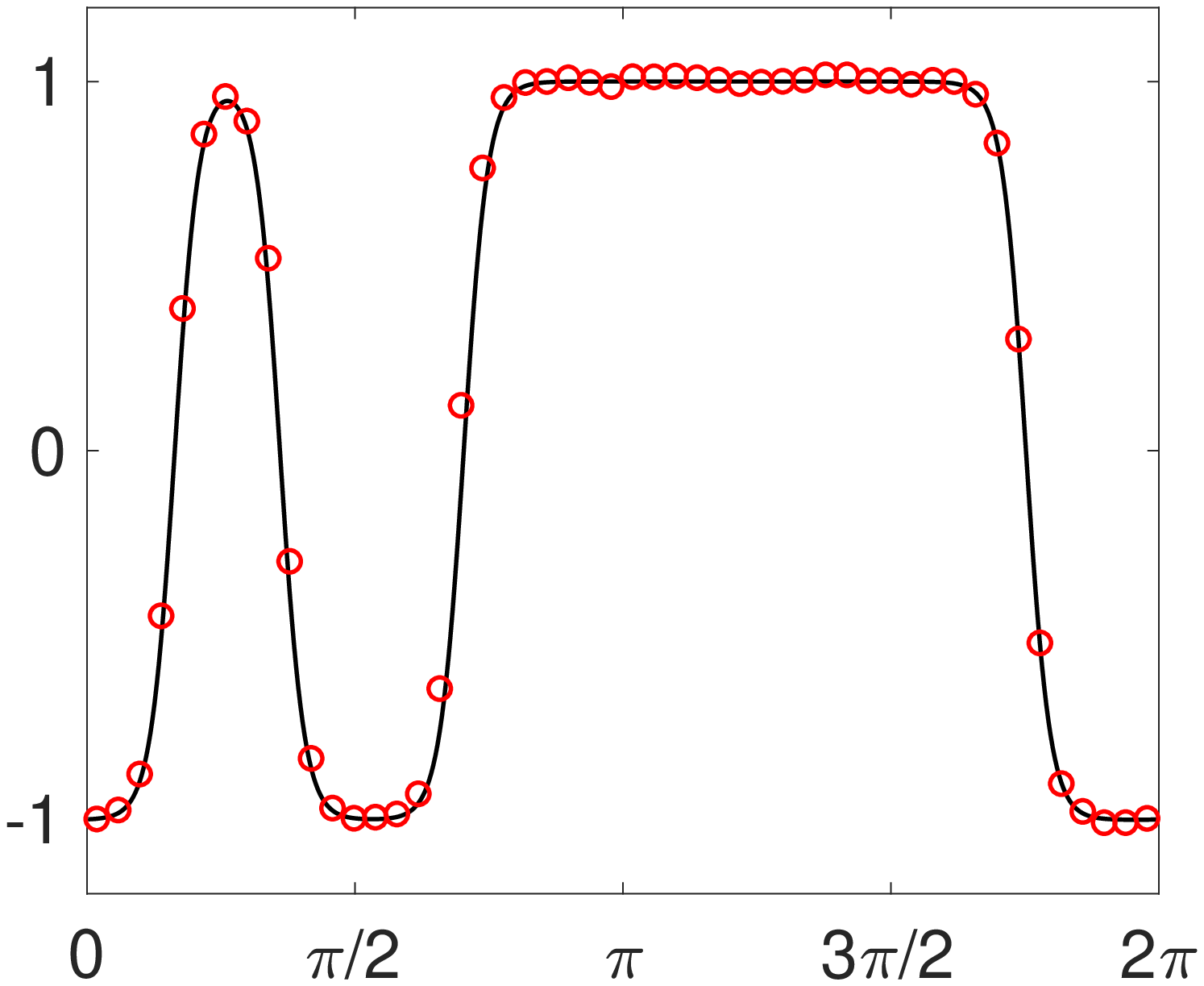}}
\hspace{0.0\textwidth} \subfloat[SAV/CN, $t=1$]{
\includegraphics[width=0.305\textwidth]{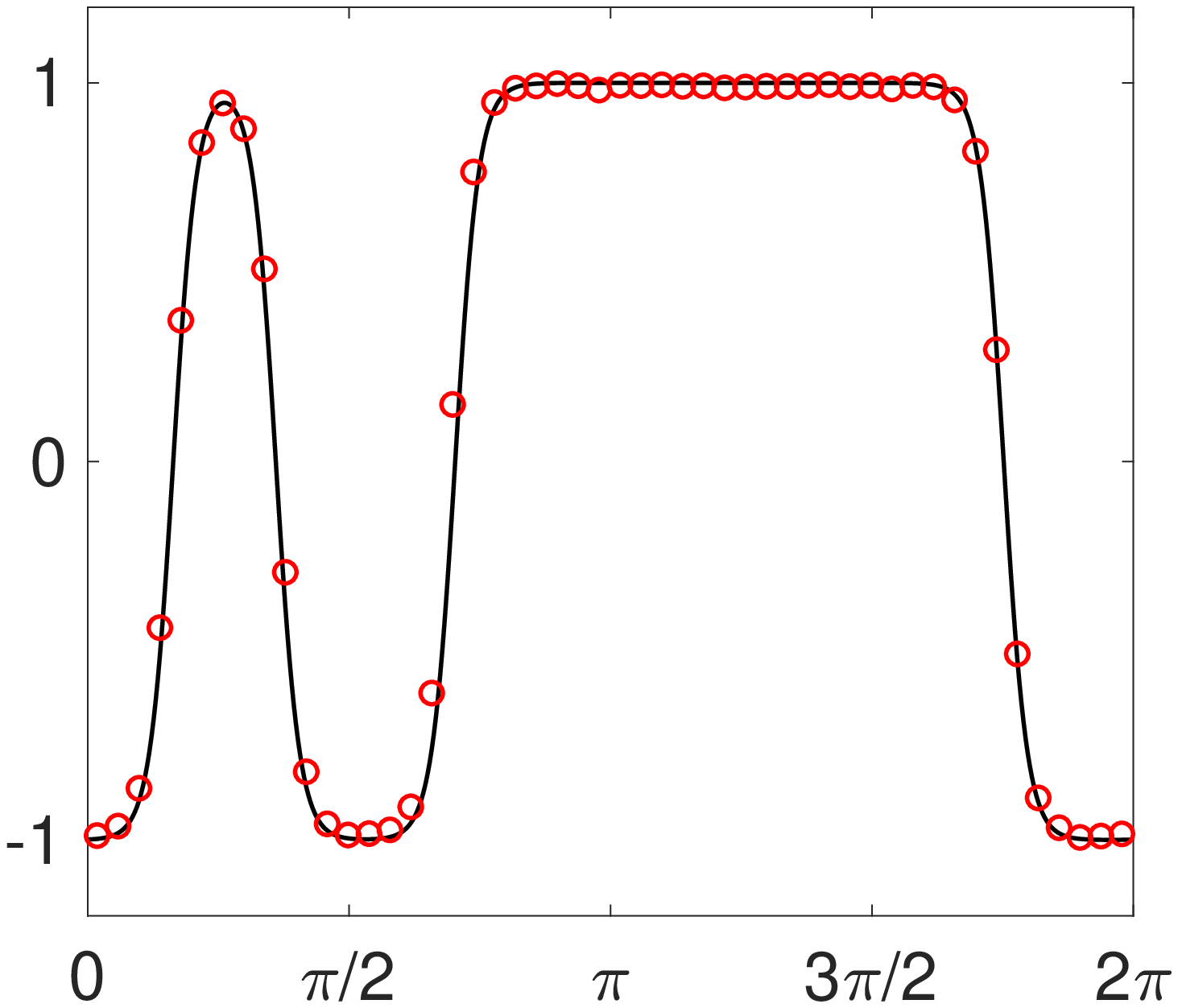}}
\hspace{0.0\textwidth} \subfloat[Sch-1, $t=1$]{
\includegraphics[width=0.305\textwidth]{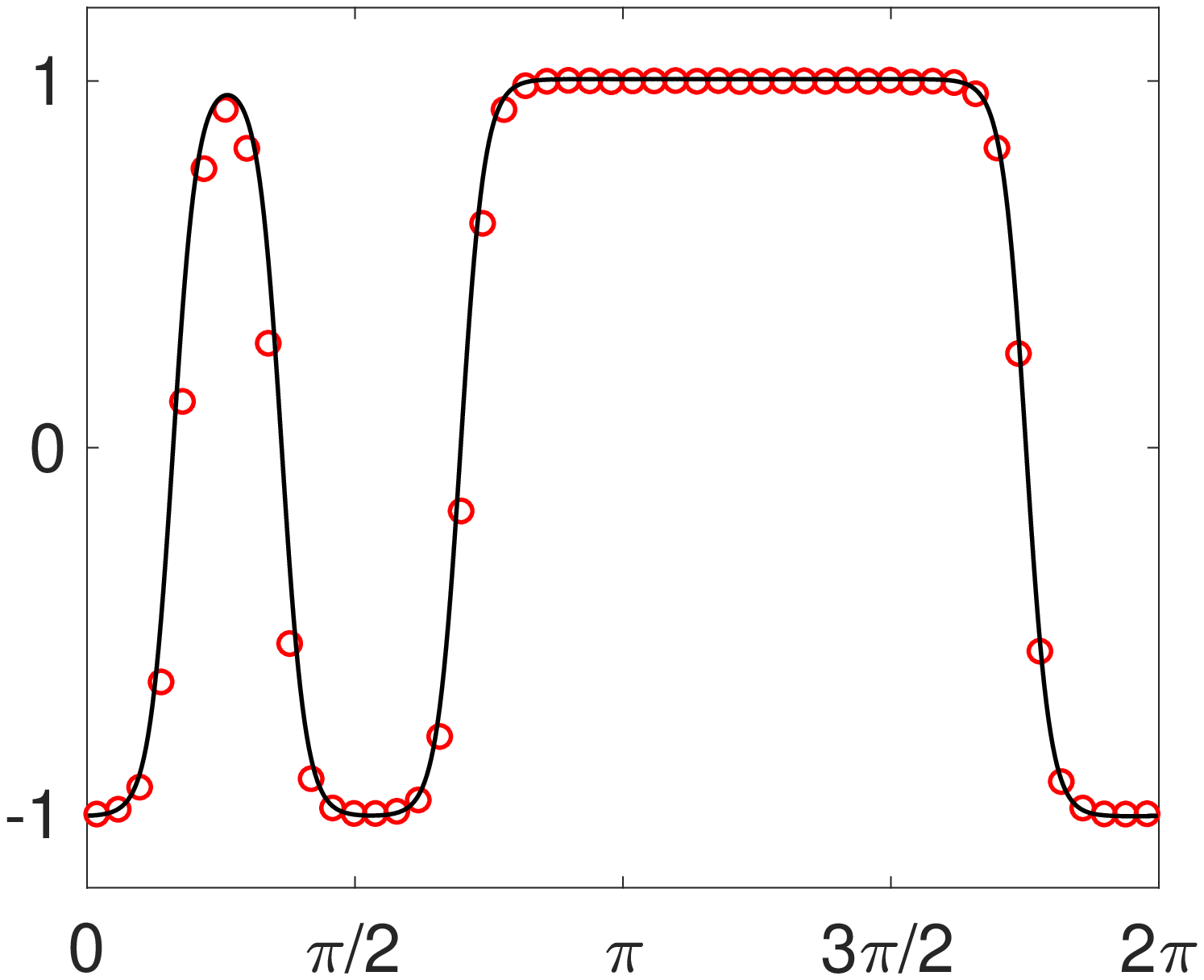}}
\hspace{0.0\textwidth} \subfloat[Sch-2, $t=1$]{
\includegraphics[width=0.305\textwidth]{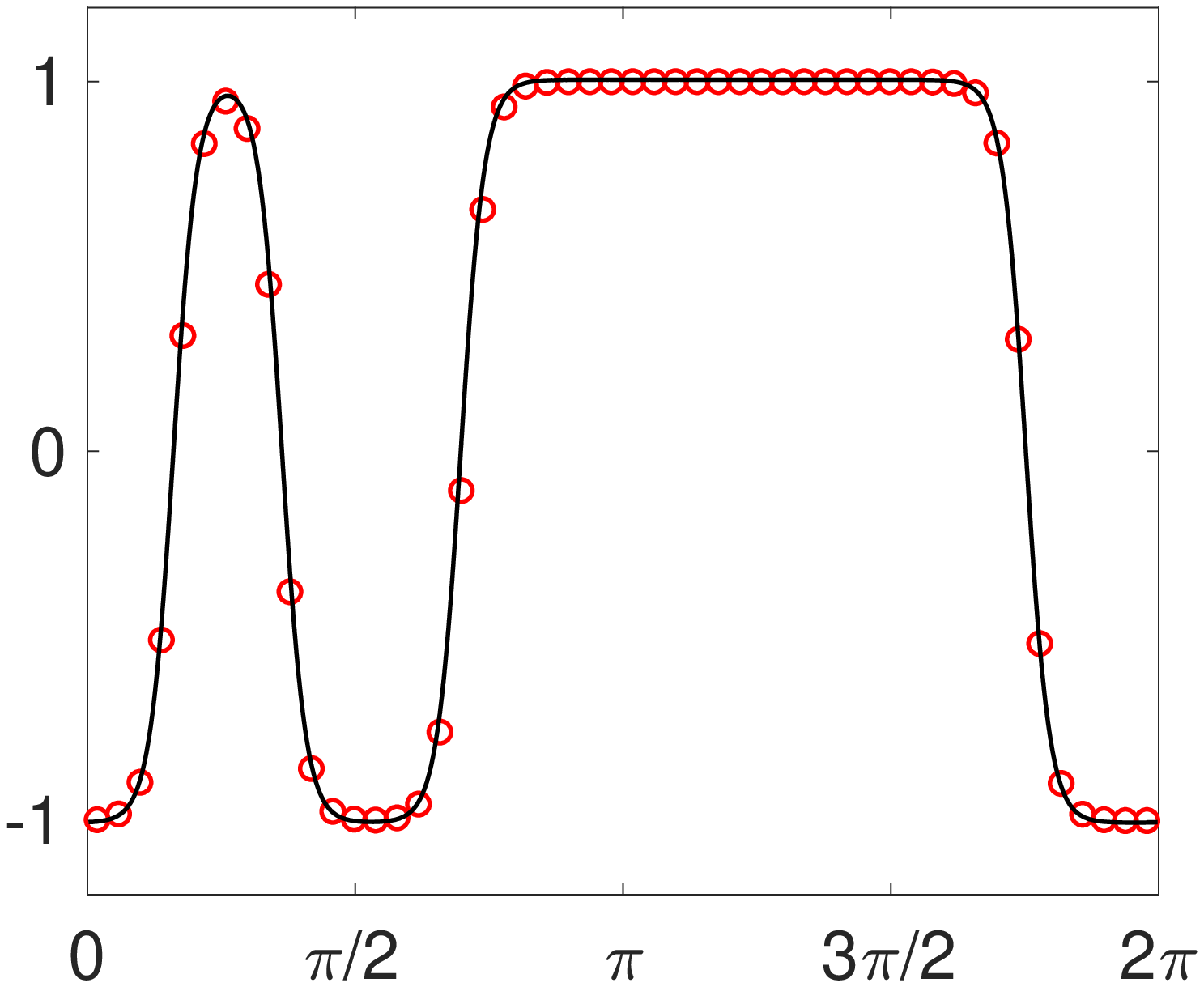}}
\hspace{0.0\textwidth} \subfloat[Sch-3, $t=1$]{
\includegraphics[width=0.305\textwidth]{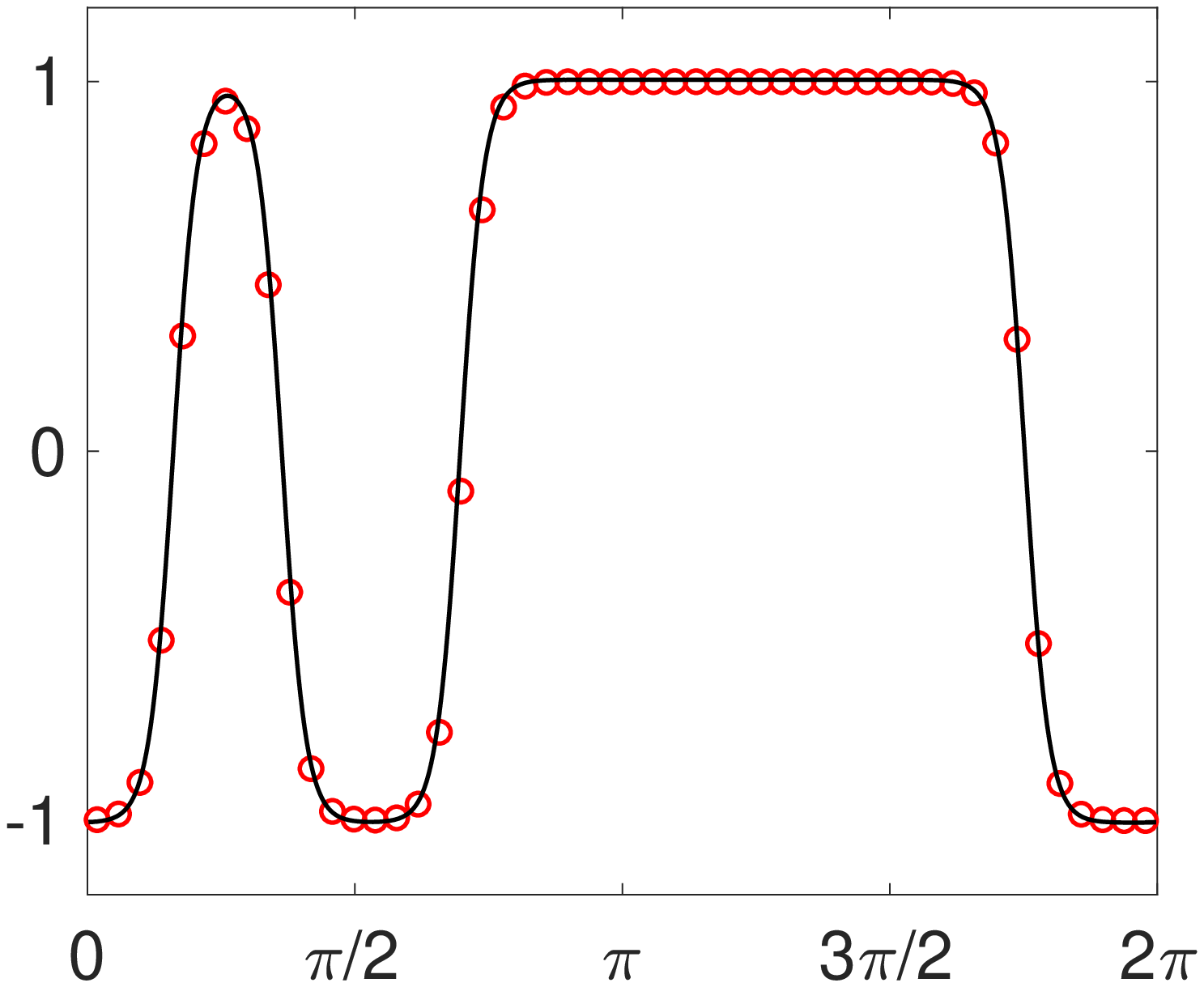}}
\hspace{0.0\textwidth} \subfloat[Sch-4, $t=1$]{
\includegraphics[width=0.305\textwidth]{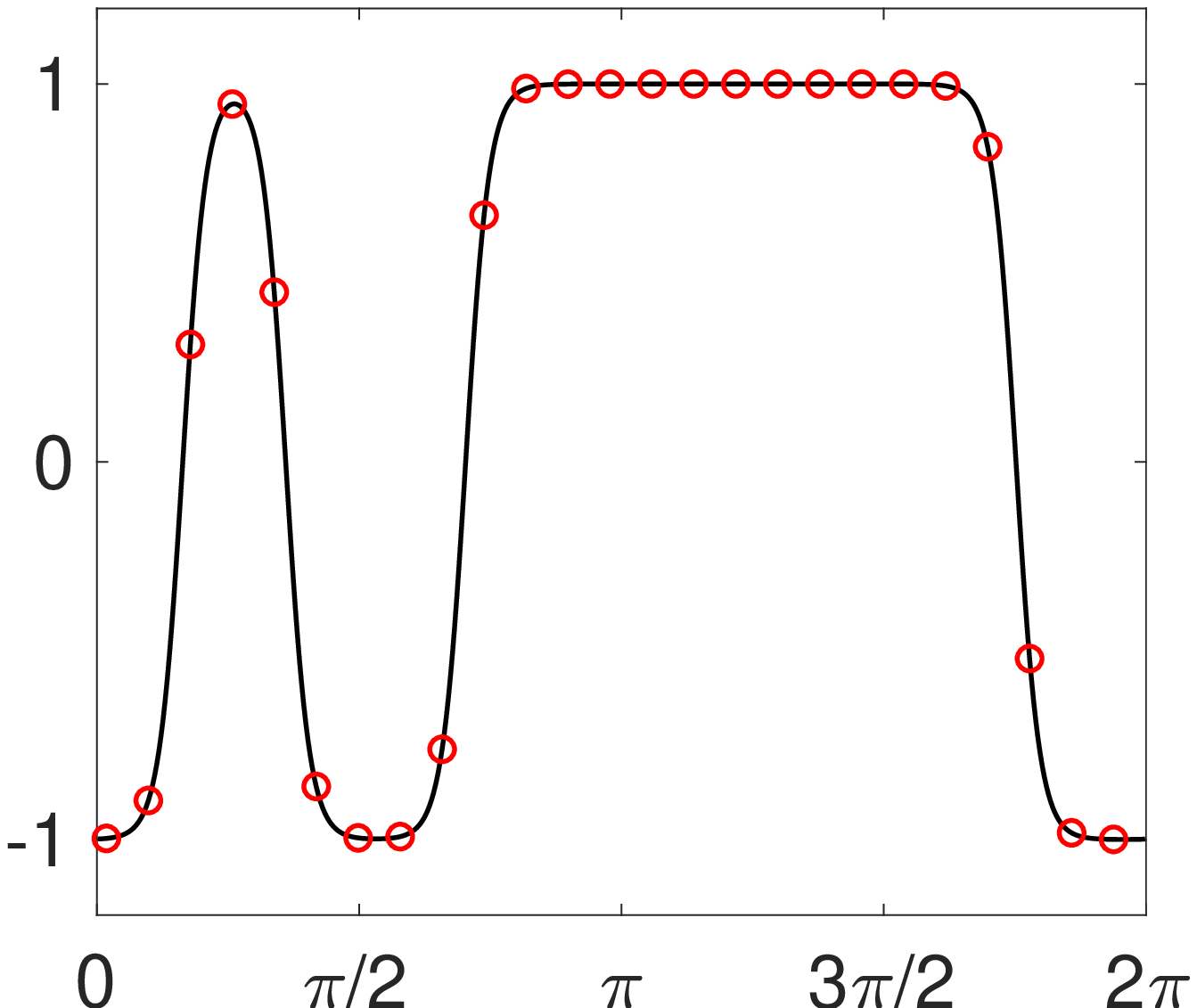}}
\caption{ Comparison of the six schemes. The red circles represent solutions obtained by the given schemes with a large time step size $h=0.01$, while the black solid lines represent the reference solutions.
}\label{randomresults}
\end{figure}

\section*{Appendix A}\label{Appendix-A}
In this appendix, we take the scheme Sch-4 as an example to illustrate how to choose the numerical constants $\tilde{a}_{ij}$.
For $\nu=3$, the constants $c_i$ are given as $c_0=0$, $c_1=1/3$, $c_2=2/3$, and $c_3=1$. The approximate solutions at $t=t_0+c_ih$ are denoted as $\phi_0$, $\phi_1$, $\phi_2$, and $\phi_3$.
The distribution of the discrete variational derivative $\mu[\phi_i,\phi_j]$ is shown in \cref{points}. 
We choose $\mu[\phi_3,\phi_0]$, $\mu[\phi_1,\phi_0]$, $\mu[\phi_2,\phi_1]$, and $\mu[\phi_3,\phi_2]$ to construct a fourth-order numerical integral for $\int_{0}^{h} \varphi dt$. 
The corresponding numerical constants are $\tilde a_{31}= \tilde a_{34}=\tilde a_{36}=3/8$, $\tilde a_{33}=-1/8$, and $\tilde a_{32}=\tilde a_{35}=0$.

For $t=1/6h,\,1/2h$, and $5/6h$, we calculate the Lagrange polynomial basic function $l_i(t)$.  The weight coefficients for the integral $\int_0^{1/3h} \varphi dt$, $\int_{0}^{2/3h} \varphi dt$, and $\int_{0}^{h} \varphi dt$ are given in the following table. 

			\begin{equation}
		\renewcommand{\arraystretch}{1.8}
	\begin{array}{ccc|c}
 	\mu[\phi_1,\phi_0]  & \mu[\phi_2,\phi_1] &\mu[\phi_3,\phi_2] & \\
 	\frac {25}{72}& -\frac 1 {36}& \frac 1 {72} & \int_0^{1/3h} \varphi dt \\
 	\frac {13} {36} & \frac{5}{18} & \frac 1{36}& \int_{0}^{2/3h} \varphi dt \\
 	\frac 3 {8} & \frac{1}{4} & \frac {3}{8}& \int_{0}^{h} \varphi dt.  
	\end{array}
	\end{equation}
	
	To obtain a fourth-order scheme satisfying the sufficient condition in \cref{theorem:0}, the numerical constants $\tilde{a}_{ij}$ for the scheme Sch-4  are finally taken as 
				\begin{equation}
		\renewcommand{\arraystretch}{1.5}
		\hbox{Sch-4},\qquad(\nu = 3,\,p=4),  \qquad
	\begin{array}{cccccc|c}
 	\frac{25}{72}& 0& - \frac 1 {24}&\frac {1} {72}  & 0 &\frac 1 {72}  & \frac 1 3 \\
 	 \frac {13}{ 36}&0  &-\frac 1{12}&\frac {13}{36}& 0 & \frac 1 {36} & \frac 2 3 \\
 	 \frac 3 8& 0 &-\frac 1 8&\frac 3 8 & 0 & \frac 3 8& 1. 
	\end{array}
	\end{equation}
	Here $(0,\,0,-\frac 1 {24}, \frac 1 {24}, 0, 0,)^T$ and  $(0,\,0,-\frac 1 {12}, \frac 1 {12}, 0, 0,)^T$  are added in the first 
	row and the second row, respectively.

	\begin{figure}[H]
	\centering{
\includegraphics[width=0.68\textwidth]{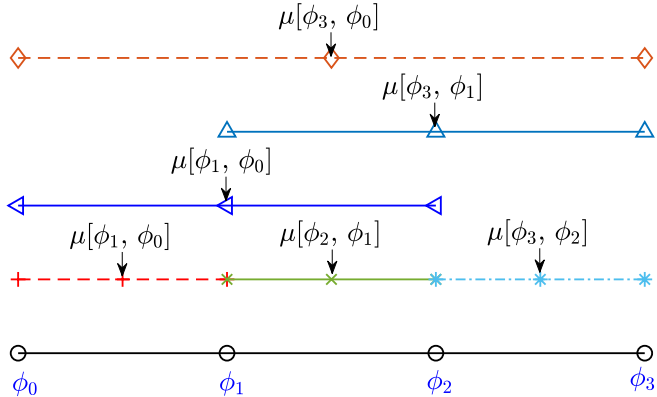}}
\caption{ Distribution of the discrete variational derivative $\mu[\phi_i,\phi_j]$. 
}\label{points}
\end{figure}



\end{document}